\documentclass[10pt]{amsart}
\usepackage{a4wide}
\usepackage{amsfonts}
\usepackage{array}
\usepackage{tabularx}
\usepackage{arydshln}
\usepackage{amsmath}

\usepackage{amssymb}
\usepackage{physics}
\usepackage{graphicx}%
\usepackage{caption}
\usepackage{subcaption}
\usepackage{multirow}
\usepackage{footmisc}
\providecommand{\U}[1]{\protect\rule{.1in}{.1in}}
\newtheorem{theorem}{Theorem}

\newtheorem{lemma}[theorem]{Lemma}

\usepackage{xcolor}
\usepackage{latexsym}
\usepackage{amsmath}
\usepackage{amssymb}
\usepackage{mathrsfs}
\usepackage{graphicx}
\usepackage{color}
\usepackage{pgfpages}
\usepackage{ifthen}
\usepackage{leftidx,tensor}
\usepackage[T1]{fontenc}
\usepackage[latin1]{inputenc}
\usepackage{mathtools}
\usepackage{comment}
\usepackage{dsfont}
\usepackage[shortlabels]{enumitem}
\usepackage{aliascnt}
\usepackage[bookmarks=true,pdfstartview=FitH, pdfborder={0 0 0}, colorlinks=true,citecolor=red, linkcolor=blue]{hyperref}
\usepackage{nicefrac}

\usepackage[textsize=tiny]{todonotes}

\usepackage[english]{babel}

\usepackage{biblatex}
\newtheorem{thm}{Theorem}[section]
\newaliascnt{cor}{thm}
\newaliascnt{prop}{thm}
\newaliascnt{lem}{thm} 

\newtheorem{prop}[prop]{Proposition}

\aliascntresetthe{cor}
\aliascntresetthe{prop}
\aliascntresetthe{lem} 
\newaliascnt{defn}{thm}
\newaliascnt{asu}{thm}
\newaliascnt{con}{thm}
\newtheorem{defn}[defn]{Definition}
\newtheorem{asu}[asu]{Assumption}
\newtheorem{con}[con]{Condition}
\aliascntresetthe{defn}
\aliascntresetthe{asu}
\aliascntresetthe{con}
\newcounter{stp}
\newcounter{stpi}
\newcounter{stpci}
\newcounter{stpiii}

\newaliascnt{rem}{thm}
\newaliascnt{exa}{thm}
\newaliascnt{masu}{thm}
\newaliascnt{nota}{thm}
\newaliascnt{sett}{thm}
\newtheorem{rem}[rem]{Remark}

\aliascntresetthe{rem}
\aliascntresetthe{exa}
\aliascntresetthe{masu}
\aliascntresetthe{nota}
\aliascntresetthe{sett}
\numberwithin{equation}{section}

\setlist[enumerate]{font = \normalfont}

\newcommand {\N}	{\mathbb{N}}

\newcommand {\R}	{\mathbb{R}}

\newcommand {\T}	{\mathbb{T}}

\newcommand {\mdsone}	{\mathds{1}}

\renewcommand{\d}{\, \mathrm{d}}

\DeclareMathOperator{\mdiv}{div}

\newcommand{\loc}{\mathrm{loc}}

	\renewcommand{\phi}{\varphi}
	
	\renewcommand{\div}{\mathrm{div} \, }
	\newcommand{\rC}{\mathrm{C}}
	\newcommand{\rL}{\mathrm{L}}

	\newcommand{\rB}{\mathrm{B}}

	\newcommand{\rX}{\mathrm X}

\newcommand{\vel}{\mathtt{u}} %or \mathsf?
\newcommand{\dir}{\mathtt{d}} %or \mathsf?
\newcommand{\veltwo}{\mathtt{v}} %or \mathsf?
\newcommand{\dirtwo}{\mathtt{f}} %or \mathsf?
\newcommand{\Dom}{\mathcal{O}}
\newcommand{\proj}{\mathbb{P}}
\renewcommand{\L}{\mathcal{L}}

\newcommand{\warning}[1]{\color{red} WARNING: #1 \color{black}}
\newcommand{\QED}{\begin{flushright} $\square$ \end{flushright}}

\title[Well-posedness and regularity for a liquid crystal model]{Well-posedness and regularity for a modified liquid crystal model with transport noise}

\author{Gianmarco Del Sarto}
\address{Technische Universit\"{a}t Darmstadt\\
Fachbereich Mathematik\\
	Schlossgartenstr.\ 7\\
	64289 Darmstadt\\
	Germany}
\email{delsarto@mathematik.tu-darmstadt.de}

\author{Antoine Moneyron}
\address{Univ Rennes\\
INRIA\\
IRMAR UMR 6625\\
Campus de Beaulieu\\
263 avenue du G\'en\'eral Leclerc\\
35042 Rennes Cedex\\
France}
\email{antoine.moneyron@inria.fr}

\author{Esm\'ee Theewis}
\address{Delft University of Technology\\
Delft Institute of Applied Mathematics\\
P.O. Box 5031\\ 
2600 GA Delft\\
The Netherlands}
\email{e.s.theewis@tudelft.nl}

\author{Tarek Z\"{o}chling}
\address{Technische Universit\"at Darmstadt\\
	Fachbereich Mathematik\\
	Schlossgartenstr.\ 7\\
	64289 Darmstadt\\
	Germany}
\email{zoechling@mathematik.tu-darmstadt.de}

\date{\today}

\begin{document}

\maketitle

\begin{abstract}
We prove well-posedness for a stochastic Lin-Liu model of modified liquid crystals, coupling an incompressible fluid to a director field under transport and multiplicative noise. On smooth bounded domains in two dimensions, we prove global well-posedness for energy-class initial data, provided the leading coefficient of the polynomial director potential
is positive. On the torus, we establish local well-posedness in
time-weighted $L^p$-$L^q$ frameworks with Besov initial data, including regimes of low and negative regularity. 
\end{abstract}

\section{Introduction}
 
\noindent
In this chapter, we study a stochastic version of a simplified model for the dynamics of nematic liquid crystals. The model couples an incompressible fluid velocity field with a director field describing the preferred orientation of the molecules. Our aim is to formulate the system as an abstract stochastic evolution equation and to prove well-posedness results in both a two-dimensional $L^2$-framework and an $L^p$-$L^q$ framework on the torus.
\subsection{Description of the model}

In the deterministic setting, the dynamics of liquid crystals are described by the Ericksen--Leslie theory, derived independently by Ericksen and Leslie from conservation laws and nematic modelling; see \cite{ericksen1961conservation,leslie1968some}. In \cite{linliu1995}, Lin and Liu introduced a simplified version of the Ericksen--Leslie system, which retains the main coupling between the fluid velocity and the molecular orientation while neglecting some stretching and rotational effects. Throughout this chapter, $\vel$ denotes the incompressible fluid velocity, $\dir$ denotes the director field, and $p$ denotes the pressure. For some $k\in\N$, $k\geq1$, and coefficients $(a_i)_{0\leq i\leq k}\in\R^{k+1}$, we set
\begin{equation}
    \Phi(d) \coloneqq (\sum_{i=0}^k a_i |d|^{2i})d. \label{eq:definition-Phi}
\end{equation}
The deterministic Lin--Liu system considered in \cite{linliu1995} is then given by
\begin{equation}\label{eq:Lin-Liu}\tag{Lin-Liu}
  \begin{cases}
    &\d \vel=\big(\Delta \vel - (\vel \cdot \nabla)\vel -\nabla p-\nabla\cdot (\nabla\dir\odot\nabla\dir) \big)\d t,\\
    & \d \dir=  \big(\Delta \dir  - (\vel \cdot \nabla)\dir-\Phi(\dir) \big)\d t,\\
    &\nabla\cdot \vel =0,
  \end{cases}
\end{equation}
where we use the notation
\begin{equation}
    \nabla\dir\odot\nabla\dir \coloneqq \sum_{i=1}^d \nabla \dir_i \otimes \nabla \dir_i.
\end{equation}
This system can be viewed as a simplified Ericksen-Leslie model in which the stretching and rotation effects generated by the director field are neglected. This simplification makes the model closer to the incompressible Navier-Stokes equations and allows one to develop existence and uniqueness results using related techniques, although it also removes some features of the full physical theory; see, for instance, \cite{feireisl2012new}.

The stochastic model we study consists in the aforementioned system \eqref{eq:Lin-Liu}, where transport and multiplicative noises have been included. More explicitly, let $\mathcal{U} = \ell^2$ and let $W$ be a cylindrical Wiener process on $\mathcal{U}$, then define the transport noise terms respectively as,
\begin{equation}
     \sum_{n\geq 1} (b_{1,n}\cdot\nabla)\vel \: \d W^n \quad \text{ and } \quad \sum_{n\geq 1} (b_{2,n}\cdot\nabla)\dir \: \d W^n,
\end{equation}
where $(b_{i,n})_{n \in \N} \in \ell^2(\R^d)$, $i\in\{1,2\}$. Similarly, the multiplicative noise terms read 
\begin{equation}
     \sum_{n\geq 1} g_{1,n}(\vel,\dir) \d W^n \quad \text{ and } \quad \sum_{n\geq 1} g_{2,n}(\dir) \d W^n,
\end{equation}
with $(g_{1,n})_{n\in\N} \colon \R^d\times\R^d \to \ell^2(\R^d)$ and $(g_{2,n})_{n\in\N} \colon \R^d \to \ell^2(\R^d)$. In addition, a martingale component of pressure increment $\sum_{n \geq 1}\pi_n dW^n$ is introduced in the system. As for the deterministic contribution, it is defined up to a constant, and such that
\begin{equation}
    \nabla \pi_n = (I - \mathbb P) \big[(b_{1,n}\cdot\nabla)\vel +g_{1,n}(\vel ,\dir)\big],
\end{equation}
where $\mathbb P$ is the Helmholtz-Leray projection defined below.

Moreover, the system is studied on a convex bounded smooth domain $\Dom \subset \R^d$ and is equipped with Dirichlet and Neumann boundary conditions on $\vel$ and $\dir$, respectively. Thus, the system of interest is the following, that we label as \eqref{eq:MLC0}.
\begin{equation}\label{eq:MLC0}\tag{MLC}
  \begin{cases}
    &\d \vel=\big(\Delta \vel - (\vel \cdot \nabla)\vel -\nabla p-\nabla\cdot (\nabla\dir\odot\nabla\dir) \big)\d t + \sum_{n\geq 1} \big[(b_{1,n}\cdot\nabla)\vel +g_{1,n}(\vel ,\dir) - \nabla \pi_n \big]\d W^n \text{ on }\Dom,\\
    & \d \dir=  \big(\Delta \dir  - (\vel \cdot \nabla)\dir-\Phi(\dir) \big)\d t + \sum_{n\geq 1}\big[(b_{2,n}\cdot\nabla)\dir+g_{2,n}(\dir)\big]\d W^n \text{ on }\Dom,\\
    &\nabla\cdot \vel =0 \text{ on }\Dom,\\
    &\vel =0 \text{ and } \partial_\nu \dir=0 \text{ on }\partial\Dom\\
    &\vel\vert_{t=0} = \vel_0 \text{ and } \dir\vert_{t=0} = \dir_0.
  \end{cases}
\end{equation}
The transport noise terms $\sum_{n \geq 1} (b_{1,n}\cdot\nabla)\vel \d W^n$ and $\sum_{n \geq 1} (b_{2,n}\cdot\nabla)\dir \d W^n$ can be thought of as perturbations induced by unresolved small scale dynamics, while the multiplicative noise terms $\sum_{n \geq 1} g_{1,n}(\vel,\dir) \d W^n$ and $\sum_{n \geq 1} g_{2,n}(\dir) \d W^n$ can be understood as external forcing contributions. Furthermore, to ensure the energy balance of the system, we assume below that the sequences $(b_{i,n})_{n \geq 1}$, $i \in \{1,2\}$, fulfil a parabolicity condition -- see \cite{agresti2025nonlinear, agresti2024stochastic} for example.

\subsection{The abstract Cauchy problem}

First, we denote by
\begin{gather}\label{eq:def projected spaces}
    \mathbb{L}^2(\Dom) = \{\vel \in L^2(\Dom,\R^d) | \mdiv \vel = 0, \: \vel \cdot \nu =0 \text{ in } H^{-1/2}(\partial\Dom;\R)\},\\
    \mathbb{H}_0^1(\Dom) = H^1_0(\Dom,\R^d) \cap \mathbb{L}^2(\Dom).
\end{gather}
Then, we consider the spaces 
\begin{equation}\label{eq: def V H}
    H = \mathbb{L}^2(\Dom) \times H^1(\Dom;\R^d) \quad \text{and} \quad V = \mathbb{H}_0^1(\Dom) \times \{\dir \in H^2(\Dom;\R^d) \,|\, \partial_\nu \dir = 0 \text{ on } \partial \Dom\},
\end{equation}
equipped with the following inner products: for all $(\vel,\dir),(\vel',\dir') \in H,$
\begin{equation}
    ((\vel,\dir),(\vel',\dir'))_H \coloneqq \int_{\Dom} (\vel \cdot \vel' + \dir \cdot \dir' + \nabla \dir : \nabla \dir') \, \dd \Dom,
\end{equation}
and, for all $(\vel,\dir),(\vel',\dir') \in V$,
\begin{equation}\label{eq: def inn prod V}
    ((\vel,\dir),(\vel',\dir'))_V \coloneqq \int_{\Dom} (\vel \cdot \vel' + \dir \cdot \dir' + \nabla \vel : \nabla \vel' + \nabla \dir : \nabla \dir' + \Delta \dir \cdot \Delta \dir') \, \dd \Dom.
\end{equation}
Then, we introduce the following duality pairing $\langle \cdot, \cdot \rangle_{V \times V'}$: for all $(\vel,\dir) \in V$ and $(\vel',\dir') \in H$,
\[
\langle (\vel, \dir), (\vel',\dir')\rangle_{V \times V'} = ((\vel, \dir), (\vel',\dir'))_H.
\]
As such, $V \subset H \subset V'$ is a Gelfand triple.

We consider $W\coloneqq (W^n)_{n\geq 1}$ to be an $\ell^2$-cylindrical Brownian motion. Moreover, we denote by $Y=(\vel,\dir)$ and write the linear operator formally as
\[
A Y=\begin{pmatrix}-\mathbb P\Delta\vel\\-\Delta\dir\end{pmatrix},
\]
where $\mathbb P\colon L^2(\Dom;\R^d)\to\mathbb L^2(\Dom)$ is the orthogonal Helmholtz projection. In the variational setting, $A\colon V\to V'$ is defined by
\begin{equation*}
\langle (\vel^\sharp,\dir^\sharp),A(\vel,\dir)\rangle_{V\times V'}
=\int_\Dom\bigl(\nabla\vel:\nabla\vel^\sharp+\nabla\dir:\nabla\dir^\sharp+\Delta\dir\cdot\Delta\dir^\sharp\bigr)\,\dd\Dom,
\end{equation*}
for $(\vel,\dir),(\vel^\sharp,\dir^\sharp)\in V$.
This defines a bounded operator from $V$ into $V'$. A director forcing $f\in L^2(\Dom;\R^d)$ is interpreted in this dual by $\langle(0,\dir^\sharp),(0,f)\rangle=\int_\Dom f\cdot(\dir^\sharp-\Delta\dir^\sharp)\,\dd\Dom$. For $f\in H^1$, integration by parts and $\partial_\nu\dir^\sharp=0$ identify this with the $H^1$ pivot pairing.

The model \eqref{eq:MLC0} can now be written as an abstract Cauchy problem in the following form:
\begin{equation}
\left \lbrace
\begin{aligned}
     \d Y(t) + A Y(t) \, dt &= F (Y(t)) \, dt + \big(B Y(t) + G(Y(t))\big) \, \d W(t), \\
     Y|_{t = 0} &= Y_0,
\end{aligned}
        \label{eq: abstract cauchy problem}
            \right. 
\end{equation}
where $Y_0 = (\vel_0, \dir_0)^T$ denotes the initial condition, and
\[
F(Y) = \begin{pmatrix}
    - \mathbb{P} \div (u \otimes u  + (\nabla\dir\odot\nabla\dir)) \\
    -(\vel \cdot \nabla)\dir - \Phi (\dir)
\end{pmatrix}, \quad B Y  = \begin{pmatrix}
   \mathbb{P}(b_{1,n}\cdot\nabla)\vel \\
   (b_{2,n}\cdot\nabla)\dir
\end{pmatrix}_{n \geq 1}, 
\quad 
G Y = \begin{pmatrix}
     \mathbb{P} g_{1,n}(\vel, \dir) \\
     g_{2,n}( \dir)
\end{pmatrix}_{n \geq 1}.
\]
In dimension two, the drift is defined on $V$ by the corresponding weak pairing
\begin{equation*}
\begin{split}
\langle(\vel^\sharp,\dir^\sharp),F(\vel,\dir)\rangle_{V\times V'}
&=\int_\Dom(\vel\otimes\vel+\nabla\dir\odot\nabla\dir):\nabla\vel^\sharp\,\dd\Dom -\int_\Dom\bigl((\vel\cdot\nabla)\dir+\Phi(\dir)\bigr)\cdot(\dir^\sharp-\Delta\dir^\sharp)\,\dd\Dom.
\end{split}
\end{equation*}
All these integrals are finite for $(\vel,\dir),(\vel^\sharp,\dir^\sharp)\in V$.

\subsection{Scaling and criticality}\label{sec:scaling}

We now investigate the scaling properties of the model \eqref{eq:Lin-Liu} and determine the corresponding critical function spaces. We consider smooth solutions on $\mathbb{R}^d$ to neglect boundary effects, and later comment on how the same heuristic applies to bounded domains.

\subsubsection*{Formal scaling analysis}
Let $\lambda > 0$ and define the rescaled fields
\begin{equation}\label{eq:rescaled_fields}
(\vel^\lambda, \dir^\lambda)(t,x) = (\lambda^{\alpha} \vel(\lambda t, \lambda^{1/2}x),\, \lambda^{\beta} \dir(\lambda t, \lambda^{1/2}x)),
\end{equation}
where $\alpha,\beta\in\mathbb{R}$ are scaling exponents to be determined. The spatial rescaling $x\mapsto \lambda^{1/2}x$ and temporal rescaling $t\mapsto \lambda t$ correspond to the standard parabolic scaling. Thus the terms of the velocity equation in \eqref{eq:Lin-Liu} have the following scaling orders: $\partial_t \vel$ and $\Delta \vel$ scale as $\lambda^{\alpha+1}$, the transport term $(\vel \cdot \nabla)\vel $ scales as $\lambda^{2\alpha + 1/2}$, the coupling term $\nabla\cdot(\nabla \dir  \odot \nabla \dir )$ scales as $\lambda^{2\beta + 3/2}$.
Balancing the diffusion and convection terms in the velocity equation yields
\begin{equation}\label{eq:scaling_rel_1}
\alpha + 1 = 2\alpha + \tfrac{1}{2},
\end{equation}
which gives $\alpha = \tfrac{1}{2}$. Balancing this with the $d$-coupling term yields
\begin{equation}\label{eq:scaling_rel_2}
2\beta + \tfrac{3}{2} = \alpha + 1 = \tfrac{3}{2}, \quad\text{hence}\quad \beta = 0.
\end{equation}
Therefore, the natural scaling exponents for the Lin--Liu system are
\begin{equation*}
(\alpha, \beta) = \big(\tfrac{1}{2}, 0\big).
\end{equation*}
With $(\alpha,\beta)=(1/2,0)$, we check that all leading-order terms in the director equation scale equally. Indeed,
\begin{equation*}
\partial_t \dir^\lambda - \Delta \dir^\lambda + (\vel^\lambda \cdot \nabla)\dir^\lambda = \lambda (\partial_t \dir - \Delta \dir + (\vel\cdot\nabla)\dir)(\lambda t, \lambda^{1/2}x).
\end{equation*}
Hence the principal deterministic terms are scale invariant under \eqref{eq:rescaled_fields}.

\subsubsection*{Scaling of stochastic perturbations}
 Define the rescaled processes $\beta^n_{t,\lambda}:=\lambda^{-1/2}W^n_{\lambda t}$, which are again standard Brownian motions by the self-similarity of $W^n$. Under this scaling, and still with the choice of $\alpha = 1/2$, we obtain
\begin{align*}
(b_{1,n}\cdot\nabla)\vel ^\lambda\,d\beta^n_{t,\lambda} &= \lambda ((b_{1,n}\cdot\nabla)\vel)(\lambda t, \lambda^{1/2}x)\,d\beta^n_{t,\lambda},\\
(b_{2,n}\cdot\nabla)\dir^\lambda\,d\beta^n_{t,\lambda} &= \lambda^{1/2}  ((b_{2,n}\cdot\nabla)\dir)(\lambda t, \lambda^{1/2}x)\,d\beta^n_{t,\lambda}.
\end{align*}
Thus, transport noise preserves the principal deterministic scaling under the Brownian time rescaling. Let $(g_{1,n})_n$ and $(g_{2,n})_n$ denote the multiplicative noise coefficients in the velocity and director equations. Formally, the associated Brownian time-change is $ \dot{\beta}^n_{t,\lambda} = \lambda^{1/2} \dot{W}^n_{\lambda t}$. The stochastic integrals transform as
\[
g_1(\vel ^\lambda,\dir ^\lambda)\d \beta^n_{t,\lambda} = \lambda^{1/2}\,g_1(\vel ^\lambda,\dir ^\lambda)\,\dot{W}^n_{\lambda t}\d t, \qquad g_2(\dir ^\lambda)\d \beta^n_{t,\lambda} = \lambda^{1/2}\,g_2(\dir ^\lambda)\,\dot{W}^n_{\lambda t}\d t.
\]
Assuming a polynomial ansatz
\[
g_{1,n}(\vel,\dir) \sim |\vel|^{\delta_1}|\dir|^{\eta_1},\qquad
g_{2,n}(\dir) \sim |\dir|^{\eta_2},
\]
and substituting $\vel^\lambda = \lambda^{1/2} \vel$, $\dir^\lambda = \dir$, one finds
\[
g_{1,n}(\vel^\lambda,\dir^\lambda)\,d\beta^n_{t,\lambda} \sim \lambda^{(1+\delta_1)/2} |\vel|^{\delta_1}|\dir|^{\eta_1} \dot{W}^n_{\lambda t}\d t, \quad
g_{2,n}(\dir^\lambda)\,d\beta^n_{t,\lambda} \sim \lambda^{1/2} |\dir|^{\eta_2} \dot{W}^n_{\lambda t}\d t.
\]
Critical or subcritical growth relative to the leading-order terms requires
\[
\frac{1+\delta_1}{2} \le 3/2,
\]
hence
\[
\delta_1 \le 2.
\]
For the director multiplicative noise, however, there is no velocity exponent anymore. Since
\[
g_{2,n}(\dir^\lambda)\d \beta^n_{t,\lambda}
\sim \lambda^{1/2},
\]
while the deterministic director equation scales like \(\lambda\), this noise is automatically subcritical from the point of view of this formal scaling. In particular, the previous restriction
\[
\delta_2\le 1
\]
disappears, because \(\delta_2\) is no longer present when \(g_{2,n}=g_{2,n}(\dir)\). 

Consequently, to remain critical or subcritical with respect to the deterministic scaling, the multiplicative noise in the velocity equation may grow at most quadratically in $\vel$. For the director equation, under the structural assumption $g_{2,n}=g_{2,n}(\dir)$, there is no velocity-growth restriction coming from scaling. The dependence on $\dir$ does not affect the formal scaling because the director has scaling exponent $0$.
\subsubsection*{Critical spaces}
Let $(u_0,d_0)$ denote the initial data, and define the rescaled data
\begin{equation*}
(\vel_{0,\lambda}, \dir_{0,\lambda})(x) = (\lambda^{1/2}\vel_0(\lambda^{1/2}x), \dir_0(\lambda^{1/2}x)).
\end{equation*}
Then for Lebesgue norms we have
\begin{equation*}
\|\vel_{0,\lambda}\|_{\rL^q(\mathbb{R}^d)} = \lambda^{\frac{1}{2}-\frac{d}{2q}} \|\vel_0\|_{\rL^q(\mathbb{R}^d)}, \qquad \|\dir_{0,\lambda}\|_{\rL^r(\mathbb{R}^d)} = \lambda^{-\frac{d}{2r}} \|\dir_0\|_{\rL^r(\mathbb{R}^d)}.
\end{equation*}
Hence scale invariance requires $q=d$ and $r=\infty$, leading to the critical Lebesgue space pair
\begin{equation*}
\rL^d(\mathbb{R}^d) \times \rL^\infty(\mathbb{R}^d).
\end{equation*}

Similarly, for homogeneous Besov norms one obtains
\begin{equation*}
\|\vel_{0,\lambda}\|_{\dot{\rB}^s_{p,q}} \sim \lambda^{\frac{1+s}{2}-\frac{d}{2p}} \|\vel_0\|_{\dot{\rB}^s_{p,q}}, \qquad \|\dir_{0,\lambda}\|_{\dot{\rB}^{s'}_{p,q}} \sim \lambda^{\frac{s'}{2}-\frac{d}{2p}} \|\dir_0\|_{\dot{\rB}^{s'}_{p,q}}.
\end{equation*}
The invariant Besov pair thus satisfies
\begin{equation*}
\frac{1+s}{2}-\frac{d}{2p}=0, \quad \frac{s'}{2}-\frac{d}{2p}=0, \quad \text{that is, } \quad s=\frac{d}{p}-1,\; s'=\frac{d}{p}.
\end{equation*}
Therefore, the scaling-invariant function space for the pair $(\vel,\dir)$ is
\begin{equation*}
\dot{\rB}^{\frac{d}{p}-1}_{p,q}(\mathbb{R}^d) \times \dot{\rB}^{\frac{d}{p}}_{p,q}(\mathbb{R}^d).
\end{equation*}

\subsection{Notation}
\label{subsec:notation}
Throughout the paper, we adopt the following notation regarding stopping times. Denote by $E$ a Banach space, $T>0$ a deterministic time, and $\tau$ a stopping time with $0\leq \tau \leq T$ $\mathbb{P}$-a.s. When we write $L^p(0,\tau;E)$ or $H^{m,p}(0,\tau;E)$, we mean the corresponding Bochner and Sobolev spaces on the random interval $(0,\tau(\omega))$. Zero extension after $\tau$ is used for Bochner $L^p$-spaces and for the integrands in stopped equations. The spaces $H^{m,p}(0,\tau;E)$ and $H^{\theta,p}(0,\tau;E)$ are restriction spaces on $(0,\tau(\omega))$, with their restriction norms.

\subsection{Definition of solutions and main assumptions}
Consider the ground spaces $X_0$ and $X_1$. Next, using the notation of \cite{agresti2025nonlinear}, we denote the real interpolation space with parameters $(\theta,p) \in (0,1) \times [1,\infty]$ by $X_{\theta,p} = (X_0, X_1)_{\theta,p}$. Similarly, we denote by $X_\theta = [X_0, X_1]_{\theta}$ the complex interpolation space of parameter $\theta \in (0,1)$. At the endpoints, $X_0$ and $X_1$ denote the ground spaces themselves. In this setting, we define a few notions of solution to \ref{eq:MLC0} as follows.
We work on a filtered probability space $(\Omega,\mathscr F,(\mathscr F_t)_{t\geq0},\mathbb P)$ satisfying the usual conditions, with $W$ an $(\mathscr F_t)$-cylindrical Brownian motion. We write $w_\kappa(t)=t^\kappa$ and $L^p_\kappa(0,T;E)=L^p((0,T),w_\kappa(t)\d t;E)$.
\begin{defn}
    Let $p \in [2,\infty)$ and $\kappa \in [0, \frac{p-2}{2}) \cup \{0\}$.
    Let $Y_0$ be an $\mathscr F_0$-measurable $X_{1-\frac{1+\kappa}{p},p}$-valued random variable.
    \begin{enumerate}
        \item A pair $(Y,\sigma)$ is called an $L_{\kappa}^p$-solution of \eqref{eq: abstract cauchy problem} if $\sigma : \Omega \rightarrow [0,\infty]$ is a stopping time, and $Y : \R_+\times\Omega \rightarrow X_0$ such that the following propositions hold,
        \begin{itemize}
            \item $Y_{\sigma \wedge \cdot}$ is a strongly progressively measurable process, and, for all $T>0$, almost surely,%
            \begin{equation*}
                Y \in L_{\kappa}^p(0,T \wedge \sigma; X_1), \quad Y_{\sigma \wedge \cdot} \in C([0,T]; X_{1-\frac{1+\kappa}{p},p}).
            \end{equation*}
            \item almost surely, for all $t\geq0$, the following equality holds in $X_0$,
            \begin{equation*}
                Y_{\sigma \wedge t} - Y_0 = \int_0^{t} \mdsone_{[0,\sigma]}(s) \Big[- AY(s) + F(Y(s))\Big] ds + \int_0^{t} \mdsone_{[0,\sigma]}(s) \Big[BY(s) + G(Y(s))\Big] \d W_s.
            \end{equation*}
        \end{itemize}
        \item A pair $(Y,\sigma)$ is called an $L_{\kappa}^p$-local solution if $\sigma$ is a stopping time and there are finite stopping times $\sigma_n\leq\sigma$ with $\sigma_n\uparrow\sigma$ almost surely such that $(Y,\sigma_n)$ is an $L_{\kappa}^p$-solution for every $n$. We call $(\sigma_n)_{n\geq1}$ a localizing sequence. It is called unique if, for every other local solution $(\tilde{Y},\tilde{\sigma})$ with the same initial datum and Brownian motion, $Y=\tilde Y$ on $[0,\sigma\wedge\tilde\sigma)$ a.s., and maximal if, in addition, $\tilde\sigma\leq\sigma$ a.s.. An $L_{\kappa}^p$-local solution $(Y,\sigma)$ is said global whenever $\sigma = \infty$ a.s.
    \end{enumerate}
\end{defn}

\noindent
Throughout the next sections, we will work with the following assumptions.

\begin{asu}\label{ass:loc well-posed2}
The following propositions hold:
\begin{enumerate}\setlength\itemsep{.4em}
    \item there exists some $\nu \in (0,2)$ such that for $i=1,2$, $(b_{i,n})_{n\in\N}\in  
     \ell^2(\R^d)$ satisfies 
     $$
     \sum_{n\in\N}(b_{i,n}\cdot\xi)^2 \leq \nu  |\xi|^2 \text{ for all $\xi\in\R^d$,  \qquad\qquad \emph{(parabolicity)}}
     $$
     \item $(g_{1,n})_{n\in\N}\colon \R^d\times\R^d\to 
     \ell^2(\R^d)$ and $(g_{2,n})_{n\in\N}\colon \R^d\to 
     \ell^2(\R^d)$ are globally Lipschitz, 
    \item The map
$g_2:\R^d\to\ell^2(\R^d)$,
$g_2(y)=(g_{2,n}(y))_{n\in\N}$,
is continuously differentiable, and its derivative
$Dg_2=(\nabla g_{2,n})_{n\in\N}:
\R^d\to\ell^2(\R^{d\times d})$
is globally Lipschitz.
\end{enumerate}
\end{asu}
\begin{asu}\label{ass:loc well-posed-glob}
Assumption \ref{ass:loc well-posed2} holds, and the coefficient $a_k$ involved in the definition of $\Phi$ (equation \ref{eq:definition-Phi}) is positive.
\end{asu}
\noindent 
As far as the multiplicative noise coefficients are concerned,
Assumption \ref{ass:loc well-posed2} is sufficient for
local well-posedness in Propositions
\ref{prop:loc well-posed L2} and
\ref{prop:loc well-posed Lp}.
For $-1<s<1$, Propositions
\ref{prop:loc well-posed Lp subcritical} and
\ref{prop:loc well-posed Lp neg} require additional
regularity of the noise coefficients and a Sobolev
embedding condition, as specified in their statements. Under an additional condition on the leading order
coefficient of $\Phi$
(Assumption \ref{ass:loc well-posed-glob}),
we also show global well-posedness in the
bidimensional $L^2$-setting.

\section{$L^2$-setting in 2D}
\label{sec: L2 setting in 2D}
%\begin{enumerate}
%    \item define ground space - Antoine
%    \item coercivity of operator - Antoine
%    \item coercivity with nonlinearity - Antoine
%\end{enumerate}

\noindent
In this section we take $X_0=V'$, $X_1=V$, $p=2$ and $\kappa=0$. Thus $X_{1/2}=H$, and by interpolation
\begin{equation*}
X_{3/4}=[H,V]_{1/2}\hookrightarrow H^{1/2}(\Dom;\R^2)\times H^{3/2}(\Dom;\R^2).
\end{equation*}
The aim of this section is to prove the well-posedness of \ref{eq: abstract cauchy problem}, that is the following proposition,
\begin{prop}[2D well-posedness]\label{prop:loc well-posed L2}
    Assume $d=2$.
    \begin{enumerate}
        \item Assume that Assumption \ref{ass:loc well-posed2} holds. For all $Y_0 \in H$, there exists a unique maximal solution $(Y, \sigma)$ to \ref{eq: abstract cauchy problem}, with $\sigma > 0$, and
        \begin{align*}
            \mathbb{P}(\sigma < \infty, \sup_{[0,\sigma)} \|Y\|_H^2 + \int_0^\sigma \|Y\|_V^2 < \infty) = 0.
        \end{align*}
        \item Assume that Assumption \ref{ass:loc well-posed-glob} holds. Then, the unique maximal solution described above is global.
    \end{enumerate}
\end{prop}
To prove the result above, we use some theorems from \cite{agresti2025nonlinear}. We remind first that the Helmholtz-Leray projector $\proj$ defined above is self-adjoint by classical arguments.

\subsection{Local well-posedness} \label{subsec:locWP}

We prove the first point of Proposition \ref{prop:loc well-posed L2} by using Theorem 6.2 of \cite{agresti2025nonlinear}. We check hereafter Condition \ref{cond:Assumption-6.1-first} and Condition \ref{cond:Assumption-6.1-second} -- see Assumption 6.1 of \cite{agresti2025nonlinear}.

\subsubsection{First condition of Assumption 6.1} The first condition we prove reads as follows,
\begin{con} \label{cond:Assumption-6.1-first}
    There exists $\theta > 0$ and $M \geq 0$ such that, for all $(\vel, \dir) \in V$,
    \begin{equation}
        \langle (\vel,\dir), A(\vel,\dir) \rangle_{V \times V'} - \frac{1}{2} \|B(\vel,\dir)\|_{\mathcal{L}_2(\mathcal{U},H)}^2 \geq \theta \|(\vel,\dir)\|_{V}^2 - M \|(\vel,\dir)\|_H^2.
    \end{equation}
\end{con}
\emph{Proof:} Let $(\vel,\dir)\in V$. By the definition of $A$,
\begin{equation*}
\langle(\vel,\dir),A(\vel,\dir)\rangle_{V\times V'}
=\|\nabla\vel\|_{L^2}^2+\|\nabla\dir\|_{L^2}^2+\|\Delta\dir\|_{L^2}^2
\geq\|(\vel,\dir)\|_V^2-\|(\vel,\dir)\|_H^2.
\end{equation*}
Summing over the standard orthonormal basis of $\mathcal U=\ell^2$, using the covariance bound in Assumption~\ref{ass:loc well-posed2} and Kadlec's formula from Lemma~\ref{lemma:Kadlec}, gives
\begin{align*}
\|B(\vel,\dir)\|_{\mathcal L_2(\mathcal U,H)}^2
&=\sum_{n\geq1}\Bigl(\|\proj(b_{1,n}\cdot\nabla)\vel\|_{L^2}^2+\|(b_{2,n}\cdot\nabla)\dir\|_{H^1}^2\Bigr)\\
&\leq\nu\Bigl(\|\nabla\vel\|_{L^2}^2+\|\nabla\dir\|_{L^2}^2+\|(\nabla\otimes\nabla)\dir\|_{L^2}^2\Bigr)\\
&\leq\nu\Bigl(\|\nabla\vel\|_{L^2}^2+\|\nabla\dir\|_{L^2}^2+\|\Delta\dir\|_{L^2}^2\Bigr).
\end{align*}
Here the coefficients $b_{i,n}$ are spatially constant and $\proj$ is an orthogonal projection. Consequently,
\begin{equation}\label{eq:locWP-cond1}
\langle(\vel,\dir),A(\vel,\dir)\rangle_{V\times V'}-\frac12\|B(\vel,\dir)\|_{\mathcal L_2(\mathcal U,H)}^2
\geq\left(1-\frac\nu2\right)\|(\vel,\dir)\|_V^2-C\|(\vel,\dir)\|_H^2.
\end{equation}

\subsubsection{Second condition of Assumption 6.1.} Moreover, the second condition that we prove is the following,
\begin{con} \label{cond:Assumption-6.1-second}
    For all $m \geq 1$, there exists a constant $L_m \geq 0$, a positive integer $r$, and a family of couples $\Big((\beta_j, \rho_j)\Big)_{1 \leq j \leq r} \in \Big((1/2,1) \times \R_+\Big)^r$ fulfilling $(2 \beta_j - 1) (\rho_j + 1) \leq 1, \forall j\in\{1,...,r\}$, so that, for all $(\vel,\dir), (\vel',\dir') \in V$ such that $\|(\vel,\dir)\|_H, \|(\vel',\dir')\|_H \leq m$,
    \begin{align}
        \|F(\vel,\dir) - F(\vel',\dir')\|_{V'} + \|G(\vel'&,\dir') - G(\vel,\dir)\|_{\mathcal{L}_2(\mathcal{U},H)}\\
        &\leq L_m \sum_{j=1}^r \|(\vel',\dir')-(\vel,\dir)\|_{X_{\beta_j}} ( 1 + \|(\vel', \dir')\|_{X_{\beta_j}}^{\rho_j} + \|(\vel,\dir)\|_{X_{\beta_j}}^{\rho_j}), \nonumber
    \end{align}
\end{con}
\emph{Proof:} Let $m\geq1$ and $(\vel,\dir),(\vel',\dir')\in V$ satisfy $\|(\vel,\dir)\|_H,\|(\vel',\dir')\|_H\leq m$. For every $(\vel^\sharp,\dir^\sharp)\in V$, the weak definition of $F$ gives
\begin{align*}
\langle(\vel^\sharp,\dir^\sharp),F(\vel,\dir)-F(\vel',\dir')\rangle_{V\times V'}
&=J_1+J_2,\\
J_1&=\int_\Dom\bigl(\vel\otimes\vel-\vel'\otimes\vel'+\nabla\dir\odot\nabla\dir-\nabla\dir'\odot\nabla\dir'\bigr):\nabla\vel^\sharp\,\dd\Dom,\\
J_2&=\int_\Dom\bigl((\vel'\cdot\nabla)\dir'-(\vel\cdot\nabla)\dir+\Phi(\dir')-\Phi(\dir)\bigr)\cdot(\dir^\sharp-\Delta\dir^\sharp)\,\dd\Dom.
\end{align*}
Using $\vel\otimes \vel - \vel' \otimes \vel' = (\vel-\vel')\otimes \vel + \vel'  \otimes (\vel - \vel')$ and H\"older's inequality, we obtain
\begin{align*}
|J_1|&\leq\|\vel'-\vel\|_{L^4}\|\nabla\vel^\sharp\|_{L^2}\bigl(\|\vel'\|_{L^4}+\|\vel\|_{L^4}\bigr)\\
&\quad+\|\nabla(\dir'-\dir)\|_{L^4}\|\nabla\vel^\sharp\|_{L^2}\bigl(\|\nabla\dir'\|_{L^4}+\|\nabla\dir\|_{L^4}\bigr),\\
|J_2|&\leq\Bigl(\|\vel'-\vel\|_{L^4}\|\nabla\dir'\|_{L^4}+\|\vel\|_{L^4}\|\nabla(\dir'-\dir)\|_{L^4}+\|\Phi(\dir')-\Phi(\dir)\|_{L^2}\Bigr)\\
&\qquad\times\bigl(\|\dir^\sharp\|_{L^2}+\|\Delta\dir^\sharp\|_{L^2}\bigr).
\end{align*}

%\todo{E: if you want you can use Rem \ref{rem:telescope young} and move the latter to Sec. 1 or 2 (I refer to it in Sec. 3)}
\noindent
In addition, the polynomial definition of $\Phi$ gives the pointwise inequality
\begin{equation*}
|\Phi(\dir')-\Phi(\dir)|\leq C|\dir'-\dir|\bigl(1+|\dir'|^{2k}+|\dir|^{2k}\bigr).
\end{equation*}
For $k\geq1$, H\"older's inequality gives
\begin{equation*}
\|\Phi(\dir')-\Phi(\dir)\|_{L^2}
\leq C\|\dir'-\dir\|_{L^4}\bigl(1+\|\dir'\|_{L^{8k}}^{2k}+\|\dir\|_{L^{8k}}^{2k}\bigr).
\end{equation*}
Gathering the estimates on $J_1$ and $J_2$, and using $H^{1/2}(\Dom)\hookrightarrow L^4(\Dom)$, the embeddings $H^1(\Dom)\hookrightarrow L^{8k}(\Dom)$ for $k\geq1$, and the embedding for $X_{3/4}$ above, yields
\begin{equation}\label{eq:locWP-cond2-estimateF}
\begin{split}
\|F(\vel,\dir)-F(\vel',\dir')\|_{V'}
&\leq C\|(\vel',\dir')-(\vel,\dir)\|_{X_{3/4}}\\
&\quad\times\bigl(1+m^{2k}+\|(\vel',\dir')\|_{X_{3/4}}+\|(\vel,\dir)\|_{X_{3/4}}\bigr).
\end{split}
\end{equation}
For the multiplicative noise, the Hilbert--Schmidt square sum and the contractivity of $\proj$ give
\begin{align*}
\|G(\vel',\dir')-G(\vel,\dir)\|_{\mathcal L_2(\mathcal U,H)}^2
&\leq\|(g_{1,n}(\vel',\dir')-g_{1,n}(\vel,\dir))_{n\geq1}\|_{L^2(\Dom;\ell^2)}^2\\
&\quad+\|(g_{2,n}(\dir')-g_{2,n}(\dir))_{n\geq1}\|_{H^1(\Dom;\ell^2)}^2.
\end{align*}
The $L^2$ terms are bounded by $C\|(\vel',\dir')-(\vel,\dir)\|_H^2$. Since $g_2$ is globally Lipschitz and differentiable, its derivative is uniformly bounded as an operator from $\R^2$ into $\ell^2(\R^2)$. The Lipschitz assumption on $\nabla g_2$, followed by the chain rule, therefore gives
\begin{align*}
&\|(\nabla g_{2,n}(\dir')\cdot\nabla\dir'-\nabla g_{2,n}(\dir)\cdot\nabla\dir)_{n\geq1}\|_{L^2(\Dom;\ell^2)}\\
&\qquad\leq C\|\dir'-\dir\|_{L^4}\|\nabla\dir'\|_{L^4}+C\|\nabla(\dir'-\dir)\|_{L^2}\\
&\qquad\leq C\|\dir'-\dir\|_{H^1}\bigl(1+\|\dir'\|_{H^{3/2}}\bigr).
\end{align*}
Consequently,
\begin{equation}\label{eq:locWP-cond2-estimateG}
\begin{split}
\|G(\vel',\dir')-G(\vel,\dir)\|_{\mathcal L_2(\mathcal U,H)}
\leq C\|(\vel',\dir')-(\vel,\dir)\|_{X_{3/4}}\bigl(1+\|(\vel',\dir')\|_{X_{3/4}}+\|(\vel,\dir)\|_{X_{3/4}}\bigr).
\end{split}
\end{equation}
Combining \eqref{eq:locWP-cond2-estimateF} and \eqref{eq:locWP-cond2-estimateG} proves Condition~\ref{cond:Assumption-6.1-second} with $L_m=C(1+m^{2k})$, $r=1$, $\beta_1=3/4$ and $\rho_1=1$, since $(2\beta_1-1)(\rho_1+1)=1$. The formulation with the real interpolation norm in \cite[Assumption~6.1]{agresti2025nonlinear} follows from $(V',V)_{3/4,1}\hookrightarrow[V',V]_{3/4}$.

As such, the conditions of Theorem 6.2 of \cite{agresti2025nonlinear} are fulfilled, which shows the existence of a maximal solution to the Cauchy problem.

\QED

%Using the fact that
%\begin{equation}
%    \nabla\cdot (\nabla\dir\odot\nabla\dir) = \sum_{i = 1}^2 \nabla\cdot (\nabla\dir_i \otimes \nabla\dir_i) = \sum_{i = 1}^2 \nabla \dir_i \Delta \dir_i + \nabla \sum_{1 \leq i,j \leq 2} \frac{1}{2} (\partial_j \dir_i)^2,
%\end{equation}

\subsection{Global well-posedness}

Assume now that Assumption \ref{ass:loc well-posed-glob} holds. To prove the second point of Proposition \ref{prop:loc well-posed L2}, we use Theorem 6.4 of \cite{agresti2025nonlinear}. Since we already proved the two conditions used in Theorem 6.2, we only need to prove the additional following one,
\begin{con}
    There exist $\theta' > 0$ and $M' \geq 0$, such that, for all $(\vel,\dir) \in V$, 
    \begin{equation}
        \langle (\vel,\dir), A(\vel,\dir) - F(\vel,\dir) \rangle_{V \times V'} - \frac{1}{2} \|B(\vel,\dir) + G(\vel,\dir)\|_{\mathcal{L}_2(\mathcal{U},H)}^2 \geq \theta' \|(\vel,\dir)\|_{V}^2 - M' (\|(\vel,\dir)\|_H^2 + 1).
    \end{equation}
\end{con}

\emph{Proof:} We first perform the following calculation for smooth fields satisfying the boundary conditions. Smooth solenoidal Dirichlet fields are dense in $\mathbb H^1_0$, and smooth Neumann fields are dense in the director component of $V$ by Lemma~\ref{lemma:Kadlec}. Then
\begin{equation*}
    \begin{split}
    \langle (\vel,\dir), - F(\vel,\dir) \rangle_{V \times V'} &=
    \Big(\proj[(\vel \cdot \nabla)\vel + \nabla \cdot (\nabla\dir\odot\nabla\dir)], \vel\Big)_{L^2(\Dom, \R^2)}
    +
    \Big((\vel \cdot \nabla)\dir + \Phi(\dir), \dir - \Delta \dir\Big)_{L^2(\Dom, \R^2)}
    \\
    &=
    \int_\Dom
    \Big[
        \nabla \cdot (\nabla\dir\odot\nabla\dir) \cdot \vel
        -
        (\vel \cdot \nabla)\dir \cdot \Delta \dir
    \Big]\,\dd\Dom
    +
    \int_{\Dom} \sum_{i=0}^k a_i |\dir|^{2i+2}\,\dd\Dom
    \\
    &\quad
    +
    \int_{\Dom} \sum_{i=0}^k a_i |\dir|^{2i} |\nabla \dir|^2\,\dd\Dom
    +
    \int_{\Dom} \sum_{i=1}^k 2i\,a_i |\dir|^{2i-2}
    \sum_{\ell=1}^d (\dir\cdot \partial_\ell \dir)^2\,\dd\Dom .
\end{split}
\end{equation*}
In addition, using the identity
\begin{equation}
    \nabla \cdot (\nabla\dir\odot\nabla\dir) = \sum_i \frac{1}{2} \nabla |\nabla \dir_i|^2 + \nabla \dir_i \Delta \dir_i,
\end{equation}
we infer that, since $u$ is divergence-free,
\begin{equation}
    \int_\Dom \Big[\nabla \cdot (\nabla\dir\odot\nabla\dir) \cdot \vel - (\vel \cdot \nabla)\dir \cdot \Delta \dir \Big]\d \Dom = \sum_i \int_\Dom \Big[(\nabla \dir_i \Delta \dir_i) \cdot \vel - (\vel\cdot \nabla\dir_i) \Delta \dir_i\Big]\d \Dom = 0.
    \label{eq: cancellation proof GWP}
\end{equation}
Furthermore, since $a_k>0$, the leading-order part of the polynomial controls the lower-order terms. More precisely, for every $\xi\in(0,a_k)$ there exists $C_\xi>0$ such that, pointwise,
\begin{equation}\label{eq:poly-coercivity-pointwise}
\begin{split}
& \sum_{i=0}^k a_i | \dir |^{2i+2} + \sum_{i=0}^k  a_i | \dir |^{2i} | \nabla \dir |^2 + \sum_{i=1}^k 2i \, a_i | \dir |^{2i-2} \sum_{ \ell = 1 }^d(\dir \cdot \partial_\ell \dir)^2 \\
&\qquad \geq (a_k-\xi)|\dir|^{2k+2} +
(a_k-\xi)| \dir |^{ 2k }| \nabla \dir |^2
- C_\xi \bigl( 1 + | \nabla \dir |^2 \bigr).
\end{split}
\end{equation}
Indeed, the leading-order terms are positive, while the lower-order terms are controlled by Young's inequality. Integrating \eqref{eq:poly-coercivity-pointwise} and using the cancellation in \eqref{eq: cancellation proof GWP}, we obtain
\begin{equation*}
\begin{split}
\langle (\vel,\dir), A(\vel,\dir) - F(\vel,\dir) \rangle_{V \times V'} &\geq \| ( \vel , \dir)\|_V^2
- \|(\vel,\dir)\|_H^2 + (a_k-\xi) \int_{\Dom}| \dir|^{2k+2}\,\dd\Dom\\ &\quad + (a_k-\xi ) \int_{\Dom} | \dir |^{ 2k }| \nabla \dir|^2 \, \dd \Dom - C_\xi \bigl ( 1 + \| \nabla \dir \|_{L^2}^2 \bigr).
\end{split}
\end{equation*}
In particular, discarding the nonnegative leading-order terms gives
\begin{equation}\label{eq:coercivity-A-minus-F}
    \langle (\vel,\dir), A(\vel,\dir) - F(\vel,\dir) \rangle_{V \times V'} \geq
    \|(\vel,\dir)\|_V^2 -  C_0 \bigl( 1+ \|(\vel, \dir )\|_H^2\bigr),
\end{equation}
for some constant $C_0>0$. Additionally, the global Lipschitz assumptions on $g_1,g_2$ and the bounded derivative of $g_2$ imply
\begin{equation*}
\|G(\vel,\dir)\|_{\mathcal L_2(\mathcal U,H)}^2\leq C\bigl(1+\|(\vel,\dir)\|_H^2\bigr).
\end{equation*}
Indeed, the $L^2$ terms have linear growth, while $\|(\nabla g_{2,n}(\dir)\cdot\nabla\dir)_{n\geq1}\|_{L^2(\Dom;\ell^2)}\leq C\|\nabla\dir\|_{L^2}$. Young's inequality in the Hilbert--Schmidt space and the direct bound on $B$ above show that, for every $\xi'>0$, there exists $M_{\xi'}>0$ such that
\begin{equation*}
\|B(\vel,\dir)+G(\vel,\dir)\|_{\mathcal L_2(\mathcal U,H)}^2
\leq(\nu+\xi')\|(\vel,\dir)\|_V^2+M_{\xi'}\bigl(1+\|(\vel,\dir)\|_H^2\bigr).
\end{equation*}
Combining this with \eqref{eq:coercivity-A-minus-F} and choosing $\xi'=1-\nu/2>0$, we obtain
\begin{equation}
\begin{split}
&\langle(\vel,\dir),A(\vel,\dir)-F(\vel,\dir)\rangle_{V\times V'}-\frac12\|B(\vel,\dir)+G(\vel,\dir)\|_{\mathcal L_2(\mathcal U,H)}^2\\
&\qquad\geq\frac{2-\nu}{4}\|(\vel,\dir)\|_V^2-M'\bigl(1+\|(\vel,\dir)\|_H^2\bigr)
\end{split}
\end{equation}
for a constant $M'\geq0$. Since $\nu<2$, the required coercivity constant $\theta'=(2-\nu)/4$ is strictly positive.

Consequently, the assumptions of the Theorem 6.4 of \cite{agresti2025nonlinear} are met, therefore there exists a unique global solution to the problem \ref{eq: abstract cauchy problem}.
    
\QED

\section{$L^p$-$L^q$-setting for the torus $\T^d$}

\noindent 
In this section, we study the liquid crystal model \eqref{eq:MLC0} in an $L^p$-$L^q$-setting. The spatial domain will be the $d$-dimensional torus, i.e. $\Dom\coloneqq\T^d$, and we will assume the same for the coefficients as in the $L^2$-setting, namely Assumption \ref{ass:loc well-posed2}. 

For the spatial variable, we utilise the following Banach  spaces, for  $s\in [-1,2)$ and $p,q>2$:
\begin{equation}\label{eq:def X0X1}
    X_0\coloneqq \mathbb{H}^{-s,q}(\mathbb{T}^d)\times H^{1-s,q}(\mathbb{T}^d;\R^d),\qquad X_1\coloneqq \mathbb{H}^{2-s,q}(\mathbb{T}^d)\times H^{3-s,q}(\mathbb{T}^d;\R^d).  
\end{equation}
Here, for $\delta\in\R$ and $q>1$, $\mathbb{H}^{\delta,q}(\mathbb{T}^d)$ is defined as 
$$
\mathbb{H}^{\delta,q}(\mathbb{T}^d)\coloneqq \{f\in {H}^{\delta,q}(\mathbb{T}^d;\R^d): \div f=0 \text{ in }\mathscr{D}'(\T^d) \}. 
$$
For details, we refer to \cite[\S 2.1]{AV24SNS}. 

For $\theta\in(0,1)$, the corresponding complex interpolation space is given by
\begin{equation}\label{eq:complexinterpolLp}
    X_\theta= \mathbb{H}^{-s+2\theta,q}(\mathbb{T}^d)\times H^{1-s+2\theta,q}(\mathbb{T}^d;\R^d) 
\end{equation}
and the trace space for $\kappa\geq 0$ is given by 
\begin{equation}\label{eq:tracespaceLp}
X_{1-\frac{1+\kappa}{p},p}= \mathbb{B}_{qp}^{-s+2-\frac{2(1+\kappa)}{p}}(\mathbb{T}^d;\R^d)\times B_{qp}^{-s+3-\frac{2(1+\kappa)}{p}}(\mathbb{T}^d;\R^d).
\end{equation}

As before, we study \eqref{eq:MLC0} as a stochastic evolution equation:
\begin{equation*} 
  \begin{cases}
    &\d (\vel,\dir)+A(\vel,\dir) \d t= F(\vel,\dir)\d t +\big(B(\vel,\dir)+G(\vel,\dir)\big) \d W,\\
     &(\vel,\dir)(0)=(\vel_0,\dir_0), 
  \end{cases}
\end{equation*}
where $W\coloneqq (W^n)_{n\geq 1}$ is  an  $\ell^2$-cylindrical Brownian motion   and  
\begin{align*}
&A(\vel,\dir)\coloneqq (-\proj\Delta \vel,-\Delta \dir),\\
&B(\vel,\dir)\coloneqq \big(\mathbb{P}(b_{1,n}\cdot\nabla)\vel, (b_{2,n}\cdot\nabla)\dir\big)_{n\geq 1},\\
 & F(\vel,\dir)\coloneqq (-\mathbb{P} \nabla \cdot  (\vel \otimes \vel  + \nabla \dir \odot \nabla \dir ), - (\vel \cdot \nabla)\dir-\Phi(\dir)),\\
 & G(\vel,\dir)\coloneqq \big(\mathbb{P} g_{1,n}(\vel ,\dir), g_{2,n}(\dir)\big)_{n\geq 1}.
\end{align*}

{In the current setting, $\proj$ appearing in $A$ denotes the bounded Helmholtz--Leray projection from $H^{\delta,q}(\T^d;\R^d)$ onto $\mathbb{H}^{\delta,q}(\T^d)$ (see \cite[\S 2.1, (3.7)]{AV24SNS}).}   
Note that the divergence free condition and the boundary conditions are accounted for by the Helmholtz--Leray projection and the choice of the ground spaces.  

Throughout what follows, we will be using the following identity in distributional sense: 
\begin{equation}\label{eq:div id}
  \text{if $\nabla\cdot \vel =0$, then }   (\vel  \cdot \nabla ) \veltwo = \nabla \cdot (\vel  \otimes \veltwo). 
\end{equation}
% Furthermore, for $A,B\in \R^{d\times d}$, we define $A\odot B\in \R^{d\times d}$ through
% \[
% (A\odot B)_{j,k}\coloneqq \sum_{i=1}^d A_{ij}B_{ik}, \qquad 1\leq j,k\leq d.
% \]
% In particular,  we have
% \[
% \nabla\dir\odot \nabla\dir=\sum_{i=1}^d \nabla \dir_i\otimes \nabla \dir_i.
% \]
\begin{rem}\label{rem:telescope young}
The following estimate will be frequently used for the $\Phi$-term appearing in $F$.
For $x,y\in\R^d$ and $m\in\N$, the mean-value formula gives
\[
\big||x|^m x-|y|^m y\big|
\leq
C_m |x-y|\big(|x|^m+|y|^m\big).
\]
Therefore,
\[
|\Phi(\dir)-\Phi(\dirtwo)|
\lesssim
|\dir-\dirtwo|
\sum_{i=0}^k\big(|\dir|^{2i}+|\dirtwo|^{2i}\big)
\lesssim
|\dir-\dirtwo|
\big(1+|\dir|^{2k}+|\dirtwo|^{2k}\big). 
\]
\end{rem}

\begin{prop}[Stochastic maximal regularity of the linear part in  \eqref{eq:MLC0}]\label{prop:SMR}
Let $s\in (-1,2)$ and $q\in[2,\infty)$. %,  $p,q\in(2,\infty)$ and $\kappa\in[0,\frac{p}{2}-1)$. 
Let $X_0$ and $X_1$ be defined by \eqref{eq:def X0X1} and let  Assumption \ref{ass:loc well-posed2} be satisfied. 
%Assume the hypotheses of Proposition~\ref{prop:loc well-posed Lp} (in particular $\Dom=\T^d$ and Assumption~\ref{ass:loc well-posed2}).
Then for every $r\in(2,\infty)$ and every $\alpha\in[0,\frac r2-1)$ the linear pair $(A,B)$ satisfies stochastic maximal regularity:
\[
(A,B)\in \mathcal{SMR}^{\bullet}_{r,\alpha}
\quad\text{in the sense of \cite[Definition~3.8]{agresti2025nonlinear}.}
\] 
\end{prop}%\todo{ E: I included $\alpha=0$, $q=2$}

\begin{proof} 
Fix $s \in (-1,2)$, $r\in(2,\infty)$ and $\alpha\in[0,\frac r2-1)$.

\noindent \textit{Velocity component.}
Set $\delta:=1-s$. Then $\delta \in (-1,2)$ and
$$
H^{-1+\delta, q} = H^{-s,q}, \qquad H^{1+\delta, q} = H^{2-s,q}.
$$
In our setting on $\T^d$, the linear velocity equation is the turbulent Stokes system with
$$
A_S=-P\Delta,   \qquad (B_{S,n}u)_{n \geq 1}=(P(b_{1,n}\cdot\nabla u))_{n \geq 1},
$$
where $P$ is the Helmholtz projection and $(b_{1,n})_{n\ge1}$ are constant divergence-free vectors
(Assumption \ref{ass:loc well-posed2}(1)).
With the choices $a^{ij}=\delta_{ij}$ and $c=h=0$, Assumption \ref{ass:loc well-posed2}(1) implies
\cite[Assumption~3.1(1)--(3)]{AV24SNS}; indeed, the stochastic parabolicity bound in
Assumption~\ref{ass:loc well-posed2}(1) is equivalent to \cite[Assumption~3.1(3)]{AV24SNS} (with $a^{ij}=\delta_{ij}$). Moreover, since the coefficients $b_{1,n}$ are constant, the coefficient-regularity assumptions in
\cite[Assumption 3.1]{AV24SNS} are satisfied for the above value of $\delta$. Hence \cite[Theorem~3.2]{AV24SNS} (with $p=r$, $\kappa=\alpha$, and this $\delta$) yields stochastic maximal
$L^r$-regularity (including the time-regularity part) for the Stokes couple in the scale
$$
X_0^u=\mathbb{H}^{-s,q},\qquad X_1^u=\mathbb{H}^{2-s,q}.
$$

\noindent\textit{Director component.} Set $\sigma:=1-s$. Then $\sigma \in (-1,2)$ and
$$
H^{\sigma, q} = H^{1-s,q}, \qquad H^{2+\sigma, q} = H^{3-s,q}.
$$
Consider the scalar operator pair
$$
A_d=-\Delta,\qquad (B_{d,n}\varphi)_{n \geq 1}=(b_{2,n} \cdot \nabla \varphi)_{n \geq 1}.
$$
%\todo{E: I replaced ref   \cite[Theorem~1.2]{AV24b} by \cite[Theorem~5.2]{AV24b}, the latter needed for the time reg part of max reg  }
By Assumption~\ref{ass:loc well-posed2}(1), and since the coefficients are constant,  the conditions in  \cite[Assumption~5.1]{AV24b} are satisfied 
(with $a^{ij}=\delta_{ij}$ and $b_n^j=b_{2,n}^j$). 
Therefore \cite[Theorem~5.2]{AV24b} (with $p=r$, $\kappa=\alpha$, and $\sigma=1-s$) yields stochastic maximal
$L^r$-regularity (including the time-regularity part) for $(-\Delta,(b_{2,n}\cdot\nabla)_n)$ on the scale
$$
X_0^d=H^{1-s,q},  \qquad   X_1^d=H^{3-s,q}.
$$

% Since the coefficients are constant, Assumption \ref{ass:loc well-posed2}(i) implies \cite[Assumption~1.1]{AV24b}
% (with $a^{ij}=\delta_{ij}$ and $b_n^j=b_{j,n}$), and in particular the stochastic parabolicity condition
% \cite[(1.2)]{AV24b} holds (again by Assumption~\ref{ass:loc well-posed2}(i)).
% Therefore \cite[Theorem~1.2]{AV24b} (with $p=r$, $\kappa=\alpha$, and $\sigma=1-s$) yields stochastic maximal
% $L^r$-regularity for $(-\Delta,(b_n\cdot\nabla)_n)$ on the scale
% $$
% X_0^d=H^{1-s,q},  \qquad   X_1^d=H^{3-s,q}.
% $$
Applying the scalar result componentwise gives the same conclusion for $\R^d$-valued directors.
Combining the velocity and director components, we obtain $(A,B)\in \mathcal{SMR}^{\bullet}_{r,\alpha}$
in the sense of \cite[Definition~3.8]{agresti2025nonlinear}.
\end{proof}

In the statements of the next three local well-posedness results,  we will distinguish between the ranges $s\in [1,2)$ and $s\in (-1,1)$, since the techniques used to handle these ranges differ substantially. 

\begin{prop}[Local well-posedness, $1\leq s<2$]\label{prop:loc well-posed Lp}
Let $s\in [ 1,2)$,  $p,q\in(2,\infty)$ and $\kappa\in[0,\frac{p}{2}-1)$. Let $X_0$ and $X_1$ be defined by \eqref{eq:def X0X1} and let  Assumption \ref{ass:loc well-posed2} be satisfied.   %\kappa range is only dictated by the (sub)criticality condition! 
Suppose that 
\begin{align}
&  1-\frac{1+\kappa}{p}< \frac{d-(s-1)q}{4q}+\frac{s}{2}\leq 1-\frac{1+\kappa}{2p} \label{eq:subcrit1}\\
&  \frac{1}{k}>s-3+\frac{d}{q}+2\big(\frac{1+\kappa}{p}\big),\label{eq:subcrit2}\\
&    (s-1)q\leq d.%\in[\frac{d}{2},\infty]. 
\end{align} 
  
% {I think here we need strict ineq. as we can not allow $q=\infty$. In that case we consider the operator on the ground space $\rL^\infty$ and this does not allow the use of any abstract tools such as maximal regularity, i think (Tarek) }

Then, for all $(\vel_0,\dir_0)\in L^0_{\mathscr F_0}(\Omega;X_{1-\frac{1+\kappa}{p},p})$, \eqref{eq:MLC0} has a unique maximal solution $(\vel,\dir,\sigma)$ with $\sigma>0$ a.s. Moreover,  a.s.: 
\begin{align*} 
(\vel,\dir)\in H_\loc^{\theta,p}(0,\sigma,w_\kappa;X_{1-\theta})\cap C([0,\sigma);X_{1-\frac{1+\kappa}{p},p})  \quad \text{ for all } \theta\in [0,1/2),
\end{align*}
where $X_{1-\theta}$ and $X_{1-\frac{1+\kappa}{p},p}$ are given by \eqref{eq:complexinterpolLp} and \eqref{eq:tracespaceLp}, respectively. 
\end{prop}

%In addition, we will also consider local well-posedness for the following range of $s$. 

\begin{prop}[Local well-posedness, $0<s<1$]\label{prop:loc well-posed Lp subcritical}
Let $s\in (0,1)$, $p,q\in(2,\infty)$ and $\kappa\in[0,\frac{p}{2}-1)$. 
Let $X_0$ and $X_1$ be defined by \eqref{eq:def X0X1} and let Assumption \ref{ass:loc well-posed2} be satisfied. Assume additionally that
\[
(g_{1,n})_{n\geq1}\in C^3(\R^d\times\R^d;\ell^2(\R^d)),
\qquad
(g_{2,n})_{n\geq1}\in C^4(\R^d;\ell^2(\R^d)),
\]
and that
\[
\frac dq<2-s-\frac{2(1+\kappa)}p.
\]

Suppose that
\begin{align}
& \frac{1+s}{2}\leq 1-\frac{2k}{1+2k}\frac{1+\kappa}{p}, \label{eq:subcrit01-1}\\
& q\geq \frac{d}{\,3-s-\frac{8k}{1+2k}\frac{1+\kappa}{p}\,}. \label{eq:subcrit01-2}
\end{align}
Then, for all $(\vel_0,\dir_0)\in L^0_{\mathscr F_0}(\Omega;X_{1-\frac{1+\kappa}{p},p})$, \eqref{eq:MLC0} has a unique maximal solution $(\vel,\dir,\sigma)$ with $\sigma>0$ a.s. Moreover, a.s.
\begin{align*}
(\vel,\dir)\in H_\loc^{\theta,p}(0,\sigma,w_\kappa;X_{1-\theta})\cap C([0,\sigma);X_{1-\frac{1+\kappa}{p},p})
\qquad \text{for all } \theta\in [0,1/2),
\end{align*}
where $X_{1-\theta}$ and $X_{1-\frac{1+\kappa}{p},p}$ are given by \eqref{eq:complexinterpolLp} and \eqref{eq:tracespaceLp}, respectively.
\end{prop}

\begin{prop}[Local well-posedness, $-1<s\leq 0$]\label{prop:loc well-posed Lp neg}
Let $s\in (-1,0]$, $p,q\in(2,\infty)$ and $\kappa\in[0,\frac{p}{2}-1)$. 
Let $X_0$ and $X_1$ be defined by \eqref{eq:def X0X1} and let Assumption \ref{ass:loc well-posed2} be satisfied. Assume additionally that
\[
(g_{1,n})_{n\geq1}\in C^3(\R^d\times\R^d;\ell^2(\R^d)),
\qquad
(g_{2,n})_{n\geq1}\in C^4(\R^d;\ell^2(\R^d)),
\]
and that
\[
\frac dq<2-s-\frac{2(1+\kappa)}p.
\]

Suppose that
\begin{align}
& 1-\frac{1+\kappa}{p}\leq 1-\frac{2k}{1+2k}\frac{1+\kappa}{p}, \label{eq:subcritneg1}\\
& q\geq \frac{d}{\,3-s-\frac{8k}{1+2k}\frac{1+\kappa}{p}\,}. \label{eq:subcritneg2}
\end{align}
Then, for all $(\vel_0,\dir_0)\in L^0_{\mathscr F_0}(\Omega;X_{1-\frac{1+\kappa}{p},p})$, \eqref{eq:MLC0} has a unique maximal solution $(\vel,\dir,\sigma)$ with $\sigma>0$ a.s. Moreover, a.s.
\begin{align*}
(\vel,\dir)\in H_\loc^{\theta,p}(0,\sigma,w_\kappa;X_{1-\theta})\cap C([0,\sigma);X_{1-\frac{1+\kappa}{p},p})
\qquad \text{for all } \theta\in [0,1/2),
\end{align*}
where $X_{1-\theta}$ and $X_{1-\frac{1+\kappa}{p},p}$ are given by \eqref{eq:complexinterpolLp} and \eqref{eq:tracespaceLp}, respectively.
\end{prop}
\noindent 
We start with the proof of Proposition \ref{prop:loc well-posed Lp}, in which we have $s\in[1,2)$, and classical techniques can be used.  

\begin{proof}[Proof of Proposition \ref{prop:loc well-posed Lp}]
 We prove that the abstract conditions in \cite[Assumption~4.1]{agresti2025nonlinear} are satisfied, so that the result follows from \cite[Theorem~4.7]{agresti2025nonlinear}. 

We note that $A\in \L(X_1,X_0)$ and $B\in\L(X_1,\gamma(\ell^2,X_{1/2}))$. %(recall that $\Dom=\mathbb{T}^d$), see \cite[\S 3.6.1]{agresti2025nonlinear} and \cite{AV24SNS}. %(3.15) in SNS paper, for \Dom=\R^d see Rem. 3.3 SNS paper
 %In particular, $\lambda_0+A$ is also sectorial on $X_0$ for some $\lambda_0\geq 0$.  %by  {\color{purple}   [add refs],  see also \cite[p.\ 43]{AV24SNS}: $1+\mathbb{S}$ has bdd $H^\infty$-calculus of angle 0}. 
 Moreover,   $\lambda+A$ is sectorial for a $\lambda\geq 0$ (see the proof of \cite[Theorem 3.2]{AV24SNS}) and $(A,B)\in \mathcal{SMR}_{p,\kappa}^{\bullet}$ by Proposition \ref{prop:SMR}. 
 It remains to prove local Lipschitz estimates for $F$ and $G$. 

We estimate $F$ term by term. We have for all $(\vel,\dir),(\veltwo,\dirtwo)\in X_1$:
\begin{align*}
     \| (\vel \cdot \nabla)\dir- (\veltwo \cdot \nabla)\dirtwo\|_{H^{1-s,q}}&\leq \|((\vel-\veltwo)\cdot \nabla) \dir\|_{H^{1-s,q}}+ \|(\veltwo \cdot \nabla)(\dir-\dirtwo)\|_{H^{1-s,q}}\\
     &\lesssim \|((\vel-\veltwo)\cdot \nabla) \dir\|_{L^{\lambda}}+ \|(\veltwo \cdot \nabla)(\dir-\dirtwo)\|_{L^{\lambda}}\\
     &\lesssim \| \vel-\veltwo\|_{L^{2\lambda}} \|\nabla \dir\|_{L^{2\lambda}}+ \| \veltwo\|_{L^{2\lambda}} \|\nabla(\dir-\dirtwo)\|_{L^{2\lambda}}\\
     &\lesssim \| \vel-\veltwo\|_{\mathbb{H}^{-s+2\beta,q}} \| \dir\|_{H^{1-s+2\beta,q}}+ \| \veltwo\|_{\mathbb{H}^{-s+2\beta,q}} \| \dir-\dirtwo \|_{H^{1-s+2\beta,q}}, %\\     &\lesssim \| (\vel-\veltwo,\dir-\dirtwo)\|_{X_\beta} \big(\| (\vel,\dir)\|_{X_\beta}+\| (\veltwo,\dirtwo)\|_{X_\beta}\big)
\end{align*}
where $\lambda \coloneqq \frac{dq}{(s-1)q+d}$   and   $\beta\coloneqq \frac{d-(s-1)q}{4q}+\frac{s}{2}\in[\frac12,1)$ satisfy
\begin{align}\label{eq:conditions Sob emb Lp case}
-\frac{d}{\lambda} =1-s-\frac{d}{q},\qquad 
-s+2\beta-\frac{d}{q}=-\frac{d}{2\lambda}.  
%-s+2\beta-\frac{d}{q}\geq-\frac{d}{2\lambda}
\end{align}
Since $s\geq1$ and $1/\lambda=1/q+(s-1)/d$, the Sobolev embedding $L^{\lambda}\hookrightarrow H^{1-s,q}$ holds. Moreover, $(s-1)q\leq d$ implies $\lambda\geq q/2>1$, justifying the earlier application of H\"older's inequality. 
Furthermore, by the definition of $\beta$ and $\lambda$ and using that  $(s-1)q\leq d$:
\[
-s+2\beta=-\frac{d}{2\lambda}+\frac{d}{q}=  \frac{1-s}{2}+\frac{d}{2q} \geq 0,
\]  
thus $H^{-s+2\beta,q}\hookrightarrow L^{2\lambda}$. For the (sub)criticality condition to hold, we further impose 
\[
 1-\frac{1+\kappa}{p}<\beta\leq 1-\frac{1+\kappa}{2p}, %>3/4 since kappa<p/2-1 and <1 trivially
%\frac{1+\kappa}{p}\leq 2(1-\beta).
\] 
which can be rewritten as \eqref{eq:subcrit1} using the definition of $\lambda$ and $\beta$. With the same $\beta$ and $\lambda$ as above, and recalling \eqref{eq:div id}, we find similarly
\begin{align*}
     \| (\vel \cdot \nabla)\vel-  (\veltwo \cdot \nabla)\veltwo\|_{H^{-s,q}}&\leq \| (\vel-\veltwo)  \otimes \vel \|_{H^{1-s,q}}+ \| \veltwo\otimes (\vel-\veltwo)  \|_{H^{1-s,q}}\\ 
     &\lesssim \| \vel-\veltwo\|_{L^{2\lambda}} \|\vel\|_{L^{2\lambda}}+ \| \veltwo\|_{L^{2\lambda}} \|\vel-\veltwo\|_{L^{2\lambda}}\\
     &\lesssim \| \vel-\veltwo\|_{\mathbb{H}^{-s+2\beta,q}}(\| \vel\|_{\mathbb{H}^{-s+2\beta,q}}+ \| \veltwo\|_{\mathbb{H}^{-s+2\beta,q}}). %\\     &\lesssim \| (\vel-\veltwo,\dir-\dirtwo)\|_{X_\beta} \big(\| (\vel,\dir)\|_{X_\beta}+\| (\veltwo,\dirtwo)\|_{X_\beta}\big),
\end{align*}
and 
%- (\vel \cdot \nabla)\vel -\nabla\cdot (\nabla\dir\odot\nabla\dir)
\begin{align*}
     \| \nabla\cdot (\nabla\dir\odot\nabla\dir)-  \nabla\cdot (\nabla\dirtwo\odot\nabla\dirtwo)\|_{H^{-s,q}} 
     &\lesssim \| \nabla\dir-\nabla\dirtwo\|_{L^{2\lambda}} \|\nabla\dir\|_{L^{2\lambda}}+ \| \nabla\dirtwo\|_{L^{2\lambda}} \|\nabla\dir-\nabla\dirtwo\|_{L^{2\lambda}}\\
     &\lesssim \| \dir-\dirtwo\|_{{H}^{1-s+2\beta,q}}(\| \dir\|_{{H}^{1-s+2\beta,q}}+ \| \dirtwo\|_{{H}^{1-s+2\beta,q}}).
     % \|((\vel-\veltwo)\cdot \nabla) \dir\|_{H^{1-s,q}}+ \|(\veltwo \cdot \nabla)(\dir-\dirtwo)\|_{H^{1-s,q}}\\
     % &\lesssim \|((\vel-\veltwo)\cdot \nabla) \dir\|_{L^{\lambda}}+ \|(\veltwo \cdot \nabla)(\dir-\dirtwo)\|_{L^{\lambda}}\\
     % &\lesssim \| \vel-\veltwo\|_{L^{2\lambda}} \|\nabla \dir\|_{L^{2\lambda}}+ \| \veltwo\|_{L^{2\lambda}} \|\nabla(\dir-\dirtwo)\|_{L^{2\lambda}}\\
     % &\lesssim \| \vel-\veltwo\|_{\mathbb{H}^{-s+2\beta,q}} \| \dir\|_{H^{1-s+2\beta,q}}+ \| \veltwo\|_{\mathbb{H}^{-s+2\beta,q}} \| \dir-\dirtwo \|_{H^{1-s+2\beta,q}}\\
     % &\lesssim \| (\vel-\veltwo,\dir-\dirtwo)\|_{X_\beta} \big(\| (\vel,\dir)\|_{X_\beta}+\| (\veltwo,\dirtwo)\|_{X_\beta}\big)
\end{align*}
Lastly, using Remark \ref{rem:telescope young}, Sobolev embeddings and H\"older's inequality, we have  
\begin{align*}
    \|\Phi(\dir)-\Phi(\dirtwo)\|_{H^{1-s,q}}&\lesssim \| | \dir - \dirtwo |(1+|\dir|^{2k}+|\dirtwo|^{2k})\|_{L^{\lambda}}\\   %\| |\dir|^{2k}\dir-|\dirtwo|^{2k}\dirtwo \|_{L^{\lambda}}\\
    %&\lesssim\| \dir - \dirtwo\|_{L^{\lambda_1}} \|1+|\dir|^{2k}+|\dirtwo|^{2k}\|_{L^{\lambda_2}}\\
    &\lesssim\| \dir - \dirtwo\|_{L^{\lambda(1+2k)}} \big(1+\|\dir\|_{L^{\lambda(1+2k)}}^{2k}+\|\dirtwo\|_{L^{\lambda(1+2k)}}^{2k}\big)\\
    &\lesssim \| \dir - \dirtwo\|_{{H}^{1-s+2\hat{\beta},q}} \big(1+\|\dir\|_{{H}^{1-s+2\hat{\beta},q}}^{2k}+\|\dirtwo\|_{{H}^{1-s+2\hat{\beta},q}}^{2k}\big),
\end{align*}
provided that
$1-s+2\hat{\beta}-\frac{d}{q}\geq -\frac{d}{\lambda(1+2k)}.$ 
Recalling \eqref{eq:conditions Sob emb Lp case}, the latter is equivalent to 
\begin{equation}\label{eq:beta1}
\hat\beta\geq\frac{k}{1+2k}\Big(s-1+\frac dq\Big).
\end{equation}
Moreover, the (sub)criticality conditions for the polynomial term ($\hat\rho=2k$) require
\begin{equation}\label{eq:beta2}
    1-\frac{1+\kappa}{p}<\hat\beta\leq1-\frac{2k}{1+2k}\frac{1+\kappa}{p}. 
\end{equation}
The upper bound above is strictly greater than $1-(1+\kappa)/p$. Furthermore, Assumption~\eqref{eq:subcrit2} is equivalent to
\[
1-\frac{2k}{1+2k}\frac{1+\kappa}{p}
>\frac{k}{1+2k}\Bigl(s-1+\frac dq\Bigr),
\]
so there exists a $\hat\beta$ that satisfies both \eqref{eq:beta1} and \eqref{eq:beta2}. 

Combining the estimates above and recalling \eqref{eq:complexinterpolLp}, we conclude that 
\begin{equation*}
    \begin{split}
        \|F ( \vel, \dir )- F ( \veltwo, \dirtwo ) \|_{X_0} &\lesssim \| ( \vel - \veltwo, \dir-\dirtwo )\|_{X_\beta}
\left( 1+\| ( \vel, \dir )\|_{X_\beta} + \|(\veltwo , \dirtwo)\|_{X_\beta}\right)\\
&\quad+ \|( \vel - \veltwo,\dir -\dirtwo) \|_ {X_{ \hat \beta}}
\left( 1 + \|( \vel, \dir)\|_{X_{\hat\beta}}^{2k}
+ \| ( \veltwo , \dirtwo) \|_{X_{\hat\beta}}^{2k}
\right).
    \end{split}
\end{equation*}

Next, we estimate $G$. We include the Helmholtz projection in the first component and use its boundedness. For the same $\lambda$ and $\beta$ as above, we have 
\begin{align*} %\gamma(\ell^2;X_{1/2})
   \| \big(\mathbb{P}[g_{1,n}(\vel,\dir) - g_{1,n}(\veltwo,\dirtwo)]\big)_{n\geq 1}\|_{\gamma(\ell^2,\mathbb{H}^{-s+1,q})}&\lesssim \| \big(g_{1,n}(\vel,\dir) - g_{1,n}(\veltwo,\dirtwo)\big)_{n\geq 1}\|_{\gamma(\ell^2,L^\lambda)}\\
   &\lesssim \| \big(g_{1,n}(\vel,\dir) - g_{1,n}(\veltwo,\dirtwo)\big)_{n\geq 1}\|_{ L^\lambda(\mathbb{T}^d;\ell^2)}\\
   &\lesssim \| (\vel,\dir) -  (\veltwo,\dirtwo) \|_{ L^\lambda} \\%L^\lambda(\mathbb{T}^d;\R^d\times\R^d)
  % &\lesssim \| (\vel,\dir) -  (\veltwo,\dirtwo) \|_{ L^{2\lambda}} \\
   &\lesssim \| (\vel,\dir) -  (\veltwo,\dirtwo) \|_{X_\beta},
\end{align*}
where we used \cite[Proposition~9.3.2]{HNVWvolume2}  in the second line, Assumption \ref{ass:loc well-posed2}(2) in  the third line,   and in the last line we used that ${H}^{1-s+2\beta,q}\hookrightarrow{H}^{-s+2\beta,q}\hookrightarrow L^{2\lambda}\hookrightarrow L^\lambda$. 

Similarly, for the second component of $G$, we have by the chain rule and Assumption \ref{ass:loc well-posed2}(2)(3): 
\begin{align*}
    \|\big(g_{2,n}(\dir) - g_{2,n}(\dirtwo)\big)_{n\geq 1}\|_{\gamma(\ell^2,{H}^{-s+2,q})}
     &\lesssim 
    \| \big(g_{2,n}(\dir) - g_{2,n}(\dirtwo)\big)_{n\geq 1}\|_{ H^{1,\lambda}(\mathbb{T}^d;\ell^2)}\\
   &\lesssim \| \dir-  \dirtwo\|_{ L^\lambda} +\|  \nabla g_{2,n}(\dir) \cdot
    \nabla\dir  -   \nabla g_{2,n}(\dirtwo)
   \cdot\nabla\dirtwo\|_{L^\lambda(\mathbb{T}^d;\ell^2)} \\
%     &\lesssim \| (\vel,\dir) -  (\veltwo,\dirtwo) \|_{X_\beta} +\Big\| \begin{pmatrix}
%    \nabla_x g_{2,n}(\vel,\dir)-\nabla_x g_{2,n}(\veltwo,\dirtwo)\\
%    \nabla g_{2,n}(\vel,\dir)-\nabla g_{2,n}(\veltwo,\dirtwo)
%    \end{pmatrix}\cdot
%    (\nabla\vel,\nabla\dir) \Big\|_{ L^\lambda(\mathbb{T}^d;\ell^2)} 
% \\
%    &\qquad+\Big\|
%    \begin{pmatrix}
%    \nabla_x g_{2,n}(\veltwo,\dirtwo)\\
%    \nabla g_{2,n}(\veltwo,\dirtwo)
%    \end{pmatrix}\cdot
%    (\nabla\vel-\nabla\veltwo,\nabla\dir-\nabla\dirtwo) \Big\|_{ L^\lambda(\mathbb{T}^d;\ell^2)} \\
   &\lesssim \|\dir - \dirtwo  \|_{L^{\lambda}} +\|
   \nabla g_{2,n}(\dir)-\nabla g_{2,n}(\dirtwo)
   \|_{ L^{2\lambda}(\mathbb{T}^d;\ell^2)}
   \| \nabla\dir\|_{ L^{2\lambda}} 
\\
   &\qquad+\|
   \nabla g_{2,n}(\dirtwo)\|_{ L^{2\lambda}(\mathbb{T}^d;\ell^2)} 
   \|\nabla\dir-\nabla\dirtwo\|_{ L^{2\lambda}} \\
   &\lesssim \|\dir-\dirtwo\|_{L^\lambda}
   +\|\dir-\dirtwo\|_{L^{2\lambda}}
   \|\nabla\dir\|_{L^{2\lambda}}\\
   &\qquad+\|1+|\dirtwo|\|_{L^{2\lambda}}
   \|\nabla\dir-\nabla\dirtwo\|_{L^{2\lambda}}\\
   &\lesssim \| \dir - \dirtwo \|_{{H}^{1-s+2\beta,q}} \big(1+ 
   \|\dir\|_{ {H}^{1-s+2\beta,q}} 
 +\|\dirtwo\|_{{H}^{1-s+2\beta,q} } \big).
\end{align*}
Recalling \eqref{eq:complexinterpolLp},the preceding bounds yield
\begin{equation}\label{eq:lp-local-lipschitz-G}
\|G(\vel,\dir)-G(\veltwo,\dirtwo)\|_{\gamma(\ell^2,X_{1/2})}
\lesssim \| (\vel-\veltwo,\dir-\dirtwo)\|_{X_\beta}
\big(1+\| (\vel,\dir)\|_{X_\beta}+\| (\veltwo,\dirtwo)\|_{X_\beta}\big).
\end{equation}
Finally, note that $X_{\beta,1}\hookrightarrow X_{\beta}$ and $X_{\hat\beta,1}\hookrightarrow X_{\hat\beta}$, concluding the proof.
\end{proof}

\begin{rem}[Local Lipschitz and growth properties]\label{rem:lipgrowth}
Under the assumptions of  Proposition~\ref{prop:loc well-posed Lp}, the   termwise Sobolev product estimates in the proof of Proposition~\ref{prop:loc well-posed Lp} validate 
  the local Lipschitz estimate of \cite[Assumption~4.1]{agresti2025nonlinear} with $ 
(\beta_1,\rho_1)=(\beta,1)$ and $ (\beta_2,\rho_2)=(\hat\beta,2k)$, 
where $\beta$ and $\hat\beta$ are the exponents introduced in that proof.
{Moreover, let $r\in(2,\infty)$, $\alpha\in[0,r/2-1)$ and $\frac{1+\alpha}{r}\leq\frac{1+\kappa}{p}$. The upper admissibility bounds in \cite[(4.3)]{agresti2025nonlinear} remain valid, but the interpolation exponents may have to be increased. For $j=1,2$, choose $\beta'_j\geq\beta_j$ such that
$$
1-\frac{1+\alpha}{r}<\beta'_j\leq 1-\frac{\rho_j}{1+\rho_j}\frac{1+\alpha}{r}.
$$
Such choices exist because the right endpoint is no smaller than the original upper bound and is strictly larger than $1-\frac{1+\alpha}{r}$. The embeddings $X_{\beta'_j,1}\hookrightarrow X_{\beta_j,1}$ preserve the nonlinear estimates. Hence \cite[Assumption~4.1]{agresti2025nonlinear} holds in the $(r,\alpha)$-framework with these exponents. This is relevant for Section \ref{sec:regularisation}.}
Finally, there exists a constant $C>0$ such that for all $x\in X_1$ one has
\begin{equation}\label{eq:kappa0-growth_rem}
\|F(x)\|_{X_0} + \|G(x)\|_{\gamma(\mathcal U,X_{1/2})}
\le C\Bigl(1+\|x\|_{X_{\beta,1}}^{2}+\|x\|_{X_{\hat\beta,1}}^{2k+1}\Bigr),
\end{equation}
where $X_{\theta,1}=(X_0,X_1)_{\theta,1}$ denotes real interpolation. 
% The local Lipschitz bound required in \cite[Assumption~4.1]{agresti2025nonlinear}
% is exactly what is proved term-by-term in the proof of Proposition~\ref{prop:loc well-posed Lp}.
% More precisely, it holds with the two pairs
% $$
% (\beta_1, \rho_1 )=(\beta,1),\qquad    (\beta_2,\rho_2 )=(\hat \beta, 2k).
% $$
% The admissibility restriction on these pairs is precisely \cite[(4.3)]{agresti2025nonlinear};
% for fixed pairs it depends on $(r,\alpha)$ only through $(1+\alpha)/r$, hence it remains valid whenever
% $$
% \frac{1+\alpha}{r} \leq \frac {1+\kappa}{p}.
% $$
To prove this, let $x=(u,d)\in X_1$.
Applying the local Lipschitz estimates from the proof of Proposition~\ref{prop:loc well-posed Lp} with $(\veltwo,\dirtwo)=(0,0)$,  and observing that $F(0)=0$ and $G(0)\in \gamma(U,X_{1/2})$ are constant, we obtain \eqref{eq:kappa0-growth_rem}.  
%and using Assumption~\ref{ass:loc well-posed2} to control the constant terms $F(0)$ and $G(0)$, yields
% $$
% \|F(x)\|_{X_0}+\|G(x)\|_{\gamma(\mathcal U,X_{1/2})}
% \lesssim 1+\|x\|_{X_{\beta}}^{2}+\|x\|_{X_{\hat\beta}}^{2k+1}.
% $$
% Finally, $X_{\beta,1}\hookrightarrow X_{\beta}$ and $X_{\hat\beta,1}\hookrightarrow X_{\hat\beta}$, so \eqref{eq:kappa0-growth} follows.
\end{rem}

 \begin{comment}
\todo[inline]{
Paraproduct estimates ref: [Taylor Tools for PDE Chap 2 Proposition 1.1, p.\ 105]: If $s>0$, $q,q_2\in(1,\infty)$, $q_1\in(1,\infty]$ and $1/q=1/q_1+1/q_2$, then $\|fg\|_{H^{s,q}}\lesssim \|f\|_{H^{s,q_1}}\|g\|_{L^{ q_2}}+\|f\|_{L^{q_1}}\|g\|_{H^{s,q_2}}$.}
\end{comment}
\begin{proof}[Proof of Proposition \ref{prop:loc well-posed Lp subcritical}] 
Throughout the proof we repeatedly use the fractional Leibniz rule
\[
\|fg\|_{H^{1-s,q}}
\lesssim
\|f\|_{H^{1-s,a_1}}\|g\|_{L^{a_2}}
+
\|f\|_{L^{b_1}}\|g\|_{H^{1-s,b_2}},
\qquad
\frac{1}{q}=\frac{1}{a_1}+\frac{1}{a_2}
=\frac{1}{b_1}+\frac{1}{b_2},
\]
for suitable exponents $a_1,a_2,b_1,b_2\in(1,\infty)$.
We estimate $F$ term by term. Let $q_1,q_2\in(1,\infty)$ be exponents to be specified below such that
\[
\frac{1}{q}=\frac{1}{q_1}+\frac{1}{q_2}.
\]
Then, for all $(\vel,\dir),(\veltwo,\dirtwo)\in X_1$, the fractional Leibniz rule yields
\begin{align*}
     \| (\vel \cdot \nabla)\dir- (\veltwo \cdot \nabla)\dirtwo\|_{H^{1-s,q}}
     &\leq \|((\vel-\veltwo)\cdot \nabla) \dir\|_{H^{1-s,q}}+ \|(\veltwo \cdot \nabla)(\dir-\dirtwo)\|_{H^{1-s,q}}\\
     &\lesssim \| \vel-\veltwo\|_{\mathbb{H}^{1-s,q_1}} \|\nabla \dir\|_{L^{q_2}}
     + \| \vel-\veltwo\|_{L^{q_2}} \|\nabla \dir\|_{H^{1-s,q_1}}\\
     &\quad + \| \veltwo\|_{\mathbb{H}^{1-s,q_1}} \|\nabla(\dir-\dirtwo)\|_{L^{q_2}}
     + \| \veltwo\|_{L^{q_2}} \|\nabla(\dir-\dirtwo)\|_{H^{1-s,q_1}}\\
     &\lesssim \| \vel-\veltwo\|_{\mathbb{H}^{-s+2\beta,q}} \| \dir\|_{H^{1-s+2\beta,q}}
     + \| \veltwo\|_{\mathbb{H}^{-s+2\beta,q}} \| \dir-\dirtwo \|_{H^{1-s+2\beta,q}},
\end{align*}
and recalling \eqref{eq:div id}, we find similarly
\begin{align*}
     \| (\vel \cdot \nabla)\vel-  (\veltwo \cdot \nabla)\veltwo\|_{H^{-s,q}}
     &\leq \| (\vel-\veltwo)  \otimes \vel \|_{H^{1-s,q}}+ \| \veltwo\otimes (\vel-\veltwo)  \|_{H^{1-s,q}}\\
     &\lesssim \| \vel-\veltwo\|_{\mathbb{H}^{1-s,q_1}} \|\vel\|_{L^{q_2}}
     + \| \vel-\veltwo\|_{L^{q_2}} \|\vel\|_{\mathbb{H}^{1-s,q_1}}\\
     &\qquad + \| \veltwo\|_{\mathbb{H}^{1-s,q_1}} \|\vel-\veltwo\|_{L^{q_2}}
     + \| \veltwo\|_{L^{q_2}} \|\vel-\veltwo\|_{\mathbb{H}^{1-s,q_1}}\\
     &\lesssim \| \vel-\veltwo\|_{\mathbb{H}^{-s+2\beta,q}}
     \big(\| \vel\|_{\mathbb{H}^{-s+2\beta,q}}+ \| \veltwo\|_{\mathbb{H}^{-s+2\beta,q}}\big),
\end{align*}
and
\begin{align*}
     \| \nabla\cdot (\nabla\dir\odot\nabla\dir)-  \nabla\cdot (\nabla\dirtwo\odot\nabla\dirtwo)\|_{H^{-s,q}} 
     &\leq \| (\nabla\dir-\nabla\dirtwo)\odot\nabla\dir\|_{H^{1-s,q}}
     + \| \nabla\dirtwo\odot(\nabla\dir-\nabla\dirtwo)\|_{H^{1-s,q}}\\
     &\lesssim \| \nabla\dir-\nabla\dirtwo\|_{H^{1-s,q_1}} \|\nabla\dir\|_{L^{q_2}}
     + \| \nabla\dir-\nabla\dirtwo\|_{L^{q_2}} \|\nabla\dir\|_{H^{1-s,q_1}}\\
     &\quad + \| \nabla\dirtwo\|_{H^{1-s,q_1}} \|\nabla\dir-\nabla\dirtwo\|_{L^{q_2}}
     + \| \nabla\dirtwo\|_{L^{q_2}} \|\nabla\dir-\nabla\dirtwo\|_{H^{1-s,q_1}}\\
     &\lesssim \| \dir-\dirtwo\|_{{H}^{1-s+2\beta,q}}
     \big(\| \dir\|_{{H}^{1-s+2\beta,q}}+ \| \dirtwo\|_{{H}^{1-s+2\beta,q}}\big).
\end{align*}
Lastly, using Remark \ref{rem:telescope young}, define the matrix-valued polynomial
\[
\Psi(\dir,\dirtwo):=\int_0^1 D\Phi\big(\dirtwo+\theta(\dir-\dirtwo)\big)\,\mathrm d\theta.
\]
Then
\[
\Phi(\dir)-\Phi(\dirtwo)=\Psi(\dir,\dirtwo)(\dir-\dirtwo),
\qquad
|\Psi(\dir,\dirtwo)|\lesssim 1+|\dir|^{2k}+|\dirtwo|^{2k}.
\]
Let $p_0,p_1\in(1,\infty)$ be exponents to be specified below, and set
\[
\tilde q_1:=p_1,
\qquad
\tilde q_2:=\frac{p_0}{2k},
\qquad
\tilde r_1:=p_0,
\qquad
\frac{1}{\tilde r_2}:=\frac{1}{p_1}+\frac{2k-1}{p_0}.
\]
Then
\[
\frac{1}{q}=\frac{1}{\tilde q_1}+\frac{1}{\tilde q_2}
=\frac{1}{\tilde r_1}+\frac{1}{\tilde r_2}
\]
and the fractional Leibniz rule yields
\begin{align*}
    \|\Phi(\dir)-\Phi(\dirtwo)\|_{H^{1-s,q}}
    &\lesssim \| {\Psi(\dir,\dirtwo)(\dir - \dirtwo)} \|_{H^{1-s,q}}\\
    &\lesssim \|\dir - \dirtwo\|_{H^{1-s,\tilde q_1}} \|\Psi(\dir,\dirtwo)\|_{L^{\tilde q_2}}
    + \|\dir - \dirtwo\|_{L^{\tilde r_1}} \|\Psi(\dir,\dirtwo)\|_{H^{1-s,\tilde r_2}}\\
    &\lesssim \|\dir - \dirtwo\|_{H^{1-s,p_1}}
    \big(1+\|\dir\|_{L^{p_0}}^{2k}+\|\dirtwo\|_{L^{p_0}}^{2k}\big)\\
    &\quad + \|\dir - \dirtwo\|_{L^{p_0}} \|\Psi(\dir,\dirtwo)\|_{H^{1-s,\tilde r_2}}.
\end{align*}
To estimate the last term, we apply the fractional Leibniz rule repeatedly to each monomial in $\Psi(\dir,\dirtwo)$, placing one factor in $H^{1-s,p_1}$ and all remaining factors in $L^{p_0}$. Since every monomial in $\Psi$ has degree at most $2k$, we obtain
\begin{align*}
    \|\Psi(\dir,\dirtwo)\|_{H^{1-s,\tilde r_2}}
    &\lesssim 1+\big(\|\dir\|_{H^{1-s,p_1}}+\|\dirtwo\|_{H^{1-s,p_1}}\big)
    \big(1+\|\dir\|_{L^{p_0}}^{2k-1}+\|\dirtwo\|_{L^{p_0}}^{2k-1}\big).
\end{align*}
Consequently,
\begin{align*}
    \|\Phi(\dir)-\Phi(\dirtwo)\|_{H^{1-s,q}}
    &\lesssim \|\dir - \dirtwo\|_{H^{1-s,p_1}}
    \big(1+\|\dir\|_{L^{p_0}}^{2k}+\|\dirtwo\|_{L^{p_0}}^{2k}\big)\\
    &\quad + \|\dir - \dirtwo\|_{L^{p_0}}\Big[1+
    \big(\|\dir\|_{H^{1-s,p_1}}+\|\dirtwo\|_{H^{1-s,p_1}}\big)
    \big(1+\|\dir\|_{L^{p_0}}^{2k-1}+\|\dirtwo\|_{L^{p_0}}^{2k-1}\big)\Big]\\
    &\lesssim \|\dir - \dirtwo\|_{{H}^{1-s+2\beta,q}}
    \big(1+\|\dir\|_{{H}^{1-s+2\beta,q}}^{2k}+\|\dirtwo\|_{{H}^{1-s+2\beta,q}}^{2k}\big).
\end{align*}
We now specify the parameters. Choose $\beta\in[0,1)$ such that
\begin{equation}\label{eq:beta-range-s01}
\max\Big\{1-\frac{1+\kappa}{p},\,\frac{1+s}{2},\,\frac{1}{4}\big(1+s+\tfrac{d}{q}\big)\Big\}
\leq \beta\leq 1-\frac{2k}{1+2k}\frac{1+\kappa}{p}.
\end{equation}
This is possible by \eqref{eq:subcrit01-1} and \eqref{eq:subcrit01-2}.
We additionally choose $\beta>1-\frac{1+\kappa}{p}$. This is always possible because
\[
1-\frac{2k}{1+2k}\frac{1+\kappa}{p}
-\left(1-\frac{1+\kappa}{p}\right)
=\frac{1+\kappa}{p(1+2k)}>0.
\]
For the bilinear terms, we choose $q_1,q_2$ such that
\[
\frac{1}{q}=\frac{1}{q_1}+\frac{1}{q_2},
\qquad
H^{-s+2\beta,q}\hookrightarrow H^{1-s,q_1},
\qquad
H^{-s+2\beta,q}\hookrightarrow L^{q_2}.
\]
Such exponents exist precisely when
\begin{equation}\label{eq:conditions-Sob-emb-bilinear}
4\beta\geq 1+s+\frac{d}{q},
\end{equation}
that is, whenever
\[
\beta\geq \frac{1}{4}\Bigl(1+s+\frac{d}{q}\Bigr).
\]
Moreover, since $\beta\geq \frac{1+s}{2}$, we have
\[
-s+2\beta-1\geq 0,
\]
so that
\[
\|\nabla \dir\|_{L^{q_2}}+\|\nabla \dir\|_{H^{1-s,q_1}}
\lesssim
\|\dir\|_{H^{1-s+2\beta,q}}.
\]
This justifies the preceding bounds for $(\vel\cdot\nabla)\dir$, $(\vel\cdot\nabla)\vel$, and $\nabla\cdot(\nabla\dir\odot\nabla\dir)$.
For the potential term, we choose $p_0,p_1$ such that
\[
\frac{1}{q}=\frac{1}{p_1}+\frac{2k}{p_0},
\qquad
H^{1-s+2\beta,q}\hookrightarrow H^{1-s,p_1},
\qquad
H^{1-s+2\beta,q}\hookrightarrow L^{p_0}.
\]
Such exponents exist whenever
\[
\frac{k}{2k+1}\Bigl(s-1+\frac{d}{q}\Bigr)\leq \beta.
\]
This condition is weaker than the lower bound
\[
\beta\geq \frac14\Bigl(1+s+\frac{d}{q}\Bigr)
\]
appearing in \eqref{eq:beta-range-s01}, and is therefore automatically satisfied. Indeed, \eqref{eq:beta-range-s01} implies
\[
\frac14\Bigl(1+s+\frac{d}{q}\Bigr)<1,
\]
hence $s-1+\frac{d}{q}<2$. If $s-1+\frac{d}{q}\leq0$, the potential threshold is automatic because $\beta>0$. Otherwise,
\[
\frac{k}{2k+1}\Bigl(s-1+\frac{d}{q}\Bigr)
\leq \frac12\Bigl(s-1+\frac{d}{q}\Bigr)
< \frac12+\frac14\Bigl(s-1+\frac{d}{q}\Bigr)
= \frac14\Bigl(1+s+\frac{d}{q}\Bigr).
\]
The strict lower bound $\beta>1-\frac{1+\kappa}{p}$ is the one required by the Agresti--Veraar framework, whereas the upper bound
\[
\beta\leq 1-\frac{2k}{1+2k}\frac{1+\kappa}{p}
\]
is exactly the {(sub)criticality} condition for the polynomial nonlinearity $\Phi$, corresponding to $\rho=2k$.
Combining the estimates above and recalling \eqref{eq:complexinterpolLp}, we conclude that
\begin{equation}\label{eq:frac-local-lipschitz-F}
\|F(\vel,\dir)-F(\veltwo,\dirtwo)\|_{X_0}
\lesssim \| (\vel-\veltwo,\dir-\dirtwo)\|_{X_\beta}
\big(1+\| (\vel,\dir)\|_{X_\beta}^{2k}+\| (\veltwo,\dirtwo)\|_{X_\beta}^{2k}\big).
\end{equation}
For $G$, the additional restriction $d/q<2-s-2(1+\kappa)/p$ and the inequality $(1+\kappa)/p<1/2$ imply
\begin{align*}
\mathbb B^{2-s-\frac{2(1+\kappa)}p}_{q,p}
&\hookrightarrow \mathbb H^{1-s,q}\cap L^\infty(\T^d;\R^d),\\
B^{3-s-\frac{2(1+\kappa)}p}_{q,p}
&\hookrightarrow H^{2-s,q}(\T^d;\R^d)\cap L^\infty(\T^d;\R^d).
\end{align*}
The Sobolev smoothness gap in both embeddings is $1-2(1+\kappa)/p>0$. Fix $m\geq1$ and $Y=(\vel,\dir),\widetilde Y=(\veltwo,\dirtwo)\in X_1$ with trace norms at most $m$. Their values lie in a fixed bounded subset of the coefficient domains. The assumed $C^3$ and $C^4$ regularity therefore bounds the derivatives of the entire $\ell^2$-valued coefficient maps on this set. The mean-value formula and Sobolev product and composition estimates, as in the proof of \cite[Lemma~4.3]{AV24SNS}, give
\begin{align*}
&\|(g_{1,n}(\vel,\dir)-g_{1,n}(\veltwo,\dirtwo))_n\|_{H^{1-s,q}(\T^d;\ell^2)}\\
&\qquad+\|(g_{2,n}(\dir)-g_{2,n}(\dirtwo))_n\|_{H^{2-s,q}(\T^d;\ell^2)}
\leq C_m\|Y-\widetilde Y\|_{X_{1-\frac{1+\kappa}{p},p}}.
\end{align*}
By the $\gamma$-Fubini isomorphism, boundedness of $\mathbb P$ and $X_{\beta,1}\hookrightarrow X_{1-\frac{1+\kappa}{p},p}$, it follows that
\begin{equation}\label{eq:noise-trace-local-lipschitz}
\|G(Y)-G(\widetilde Y)\|_{\gamma(\ell^2,X_{1/2})}
\leq C_m\|Y-\widetilde Y\|_{X_{1-\frac{1+\kappa}{p},p}}
\lesssim C_m\|Y-\widetilde Y\|_{X_{\beta,1}}.
\end{equation}
Moreover, $G(0)\in\gamma(\ell^2,X_{1/2})$ because its coefficients are spatially constant and square summable. Together with \eqref{eq:frac-local-lipschitz-F} and $X_{\beta,1}\hookrightarrow X_\beta$, this verifies \cite[Assumption~4.1]{agresti2025nonlinear}, using the drift pair $(\beta,2k)$ and the noise pair $(\beta,0)$. The drift admissibility follows from \eqref{eq:beta-range-s01}, while the noise pair is admissible since $1-(1+\kappa)/p<\beta<1$. Proposition~\ref{prop:SMR} and \cite[Theorem~4.7]{agresti2025nonlinear} now give the unique maximal solution and its stated regularity.
\end{proof}

\begin{proof}[Proof of Proposition \ref{prop:loc well-posed Lp neg}]  
Throughout the proof we use the paraproduct estimate on $\mathbb{T}^d$ in the form
\begin{equation}\label{eq:paraproduct-Td}
\|fg\|_{H^{r,p}(\mathbb{T}^d)}
\lesssim
\|f\|_{L^{p_1}(\mathbb{T}^d)}\|g\|_{H^{r,p_2}(\mathbb{T}^d)}
+
\|g\|_{L^{r_1}(\mathbb{T}^d)}\|f\|_{H^{r,r_2}(\mathbb{T}^d)},
\end{equation}
valid for every $r>0$, $p\in(1,\infty)$, $p_1,p_2,r_1,r_2\in{(1,\infty)}$ with
\[
\frac1p=\frac1{p_1}+\frac1{p_2}=\frac1{r_1}+\frac1{r_2},
\]
see \cite{Chae:02}. We estimate $F$ term by term.
Let $q_1,q_2\in(1,\infty)$ be exponents to be specified below such that
\[
\frac1q=\frac1{q_1}+\frac1{q_2}.
\]
Then, for all $(\vel,\dir),(\veltwo,\dirtwo)\in X_1$, applying \eqref{eq:paraproduct-Td} with $r=1-s$ yields
\begin{align*}
     \| (\vel \cdot \nabla)\dir- (\veltwo \cdot \nabla)\dirtwo\|_{H^{1-s,q}}
     &\leq \|((\vel-\veltwo)\cdot \nabla) \dir\|_{H^{1-s,q}}+ \|(\veltwo \cdot \nabla)(\dir-\dirtwo)\|_{H^{1-s,q}}\\
     &\lesssim \| \vel-\veltwo\|_{\mathbb{L}^{q_1}} \|\nabla \dir\|_{H^{1-s,q_2}}
     + \|\nabla \dir\|_{L^{q_1}} \|\vel-\veltwo\|_{\mathbb{H}^{1-s,q_2}}\\
     &\quad + \| \veltwo\|_{\mathbb{L}^{q_1}} \|\nabla(\dir-\dirtwo)\|_{H^{1-s,q_2}}
     + \|\nabla(\dir-\dirtwo)\|_{L^{q_1}} \|\veltwo\|_{\mathbb{H}^{1-s,q_2}}\\
     &\lesssim \| \vel-\veltwo\|_{\mathbb{H}^{-s+2\beta,q}} \| \dir\|_{H^{1-s+2\beta,q}}
     + \| \veltwo\|_{\mathbb{H}^{-s+2\beta,q}} \| \dir-\dirtwo \|_{H^{1-s+2\beta,q}}.
\end{align*}
Recalling \eqref{eq:div id}, we find similarly
\begin{align*}
     \| (\vel \cdot \nabla)\vel-  (\veltwo \cdot \nabla)\veltwo\|_{H^{-s,q}}
     &\leq \| (\vel-\veltwo)  \otimes \vel \|_{H^{1-s,q}}+ \| \veltwo\otimes (\vel-\veltwo)  \|_{H^{1-s,q}}\\
     &\lesssim \| \vel-\veltwo\|_{\mathbb{L}^{q_1}} \|\vel\|_{\mathbb{H}^{1-s,q_2}}
     + \|\vel\|_{\mathbb{L}^{q_1}} \|\vel-\veltwo\|_{\mathbb{H}^{1-s,q_2}}\\
     &\quad + \| \veltwo\|_{\mathbb{L}^{q_1}} \|\vel-\veltwo\|_{\mathbb{H}^{1-s,q_2}}
     + \|\vel-\veltwo\|_{\mathbb{L}^{q_1}} \|\veltwo\|_{\mathbb{H}^{1-s,q_2}}\\
     &\lesssim \| \vel-\veltwo\|_{\mathbb{H}^{-s+2\beta,q}}
     \big(\| \vel\|_{\mathbb{H}^{-s+2\beta,q}}+ \| \veltwo\|_{\mathbb{H}^{-s+2\beta,q}}\big),
\end{align*}
and
\begin{align*}
     \| \nabla\cdot (\nabla\dir\odot\nabla\dir)-  \nabla\cdot (\nabla\dirtwo\odot\nabla\dirtwo)\|_{H^{-s,q}} 
     &\leq \| (\nabla\dir-\nabla\dirtwo)\odot\nabla\dir\|_{H^{1-s,q}}
     + \| \nabla\dirtwo\odot(\nabla\dir-\nabla\dirtwo)\|_{H^{1-s,q}}\\
     &\lesssim \| \nabla\dir-\nabla\dirtwo\|_{L^{q_1}} \|\nabla\dir\|_{H^{1-s,q_2}}
     + \|\nabla\dir\|_{L^{q_1}} \|\nabla\dir-\nabla\dirtwo\|_{H^{1-s,q_2}}\\
     &\quad + \| \nabla\dirtwo\|_{L^{q_1}} \|\nabla\dir-\nabla\dirtwo\|_{H^{1-s,q_2}}
     + \|\nabla\dir-\nabla\dirtwo\|_{L^{q_1}} \|\nabla\dirtwo\|_{H^{1-s,q_2}}\\
     &\lesssim \| \dir-\dirtwo\|_{{H}^{1-s+2\beta,q}}
     \big(\| \dir\|_{{H}^{1-s+2\beta,q}}+ \| \dirtwo\|_{{H}^{1-s+2\beta,q}}\big).
\end{align*}
Lastly, using Remark \ref{rem:telescope young}, define the matrix-valued polynomial
\[
\Psi(\dir,\dirtwo):=\int_0^1 D\Phi\big(\dirtwo+\theta(\dir-\dirtwo)\big)\,\mathrm d\theta.
\]
Then
\[
\Phi(\dir)-\Phi(\dirtwo)=\Psi(\dir,\dirtwo)(\dir-\dirtwo),
\qquad
|\Psi(\dir,\dirtwo)|\lesssim 1+|\dir|^{2k}+|\dirtwo|^{2k}.
\]
Let $p_0,p_1\in(1,\infty)$ be exponents to be specified below such that
\[
\frac1q=\frac1{p_1}+\frac{2k}{p_0}.
\]
Then, applying \eqref{eq:paraproduct-Td} with $r=1-s$ and estimating $\Psi$ term by term, we obtain
\begin{align*}
    \|\Phi(\dir)-\Phi(\dirtwo)\|_{H^{1-s,q}}
    &\lesssim \| {\Psi(\dir,\dirtwo)(\dir - \dirtwo)} \|_{H^{1-s,q}}\\
    &\lesssim \|\dir - \dirtwo\|_{H^{1-s,p_1}} \|\Psi(\dir,\dirtwo)\|_{L^{\frac{p_0}{2k}}}
    + \|\dir - \dirtwo\|_{L^{p_0}} \|\Psi(\dir,\dirtwo)\|_{H^{1-s,\tilde p}}\\
    &\lesssim \|\dir - \dirtwo\|_{H^{1-s,p_1}}
    \bigl(1+\|\dir\|_{L^{p_0}}^{2k}+\|\dirtwo\|_{L^{p_0}}^{2k}\bigr)\\
    &\qquad + \|\dir - \dirtwo\|_{L^{p_0}}\Big[1+
    \bigl(\|\dir\|_{H^{1-s,p_1}}+\|\dirtwo\|_{H^{1-s,p_1}}\bigr)
    \bigl(1+\|\dir\|_{L^{p_0}}^{2k-1}+\|\dirtwo\|_{L^{p_0}}^{2k-1}\bigr)\Big]\\
    &\lesssim \|\dir - \dirtwo\|_{{H}^{1-s+2\beta,q}}
    \bigl(1+\|\dir\|_{{H}^{1-s+2\beta,q}}^{2k}+\|\dirtwo\|_{{H}^{1-s+2\beta,q}}^{2k}\bigr),
\end{align*}
where $\tilde p$ is determined by
\[
\frac1{\tilde p}=\frac1{p_1}+\frac{2k-1}{p_0}.
\]
We now specify the parameters. Choose $\beta\in[0,1)$ such that
\begin{equation}\label{eq:beta-range-neg-proof}
\max\Bigl\{1-\frac{1+\kappa}{p},\,\frac14\Bigl(1+s+\frac{d}{q}\Bigr)\Bigr\}
\leq \beta\leq 1-\frac{2k}{1+2k}\frac{1+\kappa}{p}.
\end{equation}
Choose additionally $\beta>1-\frac{1+\kappa}{p}$. The interval is nonempty by \eqref{eq:subcritneg2}, and this strict choice is possible because its upper endpoint exceeds the trace index by $\frac{1+\kappa}{p(1+2k)}>0$. For the bilinear terms, we choose $q_1,q_2$ such that
\[
\frac1q=\frac1{q_1}+\frac1{q_2},
\qquad
H^{-s+2\beta,q}\hookrightarrow L^{q_1},
\qquad
H^{-s+2\beta,q}\hookrightarrow H^{1-s,q_2}.
\]
These embeddings are possible whenever
\[
2\beta\geq 1
\qquad\text{and}\qquad
4\beta\geq 1+s+\frac{d}{q}.
\]
Since $\beta\geq 1-\frac{1+\kappa}{p}>\frac12$, the first condition is automatic. Hence the bilinear terms are controlled under the lower bound
\[
\beta\geq \frac14\Bigl(1+s+\frac{d}{q}\Bigr).
\]
For the potential term, we choose $p_0,p_1$ such that
\[
\frac1q=\frac1{p_1}+\frac{2k}{p_0},
\qquad
H^{1-s+2\beta,q}\hookrightarrow H^{1-s,p_1},
\qquad
H^{1-s+2\beta,q}\hookrightarrow L^{p_0}.
\]
Such exponents exist whenever
\[
\beta\geq \frac{k}{2k+1}\Bigl(s-1+\frac{d}{q}\Bigr).
\]
This condition is weaker than the lower bound
\[
\beta\geq \frac14\Bigl(1+s+\frac{d}{q}\Bigr)
\]
appearing in \eqref{eq:beta-range-neg-proof}, and is therefore automatically satisfied. Indeed, if $s-1+d/q\leq0$ the potential threshold is immediate; otherwise use $0<s-1+d/q<2$ and the same comparison as in the preceding proof.

The strict lower bound $\beta>1-\frac{1+\kappa}{p}$ is the one required by the Agresti--Veraar framework, whereas the upper bound
\[
\beta\leq 1-\frac{2k}{1+2k}\frac{1+\kappa}{p}
\]
is exactly the {(sub)criticality} condition for the polynomial nonlinearity $\Phi$, corresponding to $\rho=2k$.
Combining the estimates above and recalling \eqref{eq:complexinterpolLp}, we conclude that
\begin{equation}\label{eq:F-local-lipschitz-neg-proof}
\|F(\vel,\dir)-F(\veltwo,\dirtwo)\|_{X_0}
\lesssim
\|(\vel-\veltwo,\dir-\dirtwo)\|_{X_\beta}
\bigl(1+\|(\vel,\dir)\|_{X_\beta}^{2k}+\|(\veltwo,\dirtwo)\|_{X_\beta}^{2k}\bigr).
\end{equation}
For $G$, the trace-space embeddings from the preceding proof remain valid under the additional restriction $d/q<2-s-2(1+\kappa)/p$. Here $1-s\in[1,2)$ and $2-s\in[2,3)$. The $C^3$ regularity of $(g_{1,n})_n$ and the $C^4$ regularity of $(g_{2,n})_n$, together with the trace-space $L^\infty$ bounds, therefore give the same composition estimate \eqref{eq:noise-trace-local-lipschitz}. Also $G(0)\in\gamma(\ell^2,X_{1/2})$. Together with \eqref{eq:F-local-lipschitz-neg-proof} and $X_{\beta,1}\hookrightarrow X_\beta$, this verifies \cite[Assumption~4.1]{agresti2025nonlinear}, using the drift pair $(\beta,2k)$ and the noise pair $(\beta,0)$. The drift admissibility follows from \eqref{eq:beta-range-neg-proof}, while the noise pair is admissible since $1-(1+\kappa)/p<\beta<1$. Proposition~\ref{prop:SMR} and \cite[Theorem~4.7]{agresti2025nonlinear} now give the unique maximal solution and its stated regularity.
\end{proof}

\section{Blow-up criteria}
\label{sec: blow-up criteria}

\noindent 
Throughout this section we work on $\Dom=\T^d$ and under Assumption~\ref{ass:loc well-posed2}. Fix $s\in[1,2)$, $p,q\in(2,\infty)$ and $\kappa\in\big[0,\frac p2-1\big)$, and assume the hypotheses of Proposition \ref{prop:loc well-posed Lp}. Let $Y=(\vel,\dir)$ and let $(Y,\sigma)$ be the corresponding maximal local solution with
$$
Y\in H^{\theta,p}_{\loc}(0,\sigma,w_\kappa;X_{1-\theta})\cap C([0,\sigma);X_{1-\frac{1+\kappa}{p},p})
\qquad\text{for all }\theta\in[0,1/2).
$$
Recall that the trace space is given in \eqref{eq:tracespaceLp}, i.e. 
$$
X_{1-\frac{1+\kappa}{p},p}
=\mathbb{B}^{-s+2-\frac{2(1+\kappa)}{p}}_{q, p}(\T^d;\R^d)\times
\rB^{-s+3-\frac{2(1+\kappa)}{p}}_{q, p}(\T^d;\R^d),
$$
and the complex interpolation spaces are given by \eqref{eq:complexinterpolLp}, i.e.
$$
X_{1-\frac{\kappa}{p}}
=\mathbb{H}^{-s+2-\frac{2\kappa}{p},q}(\T^d)\times H^{3-s-\frac{2\kappa}{p},q}(\T^d;\R^d).
$$

\begin{thm}[Blow-up criteria for \eqref{eq:MLC0} in the $L^p$--$L^q$ setting]
Let $(Y,\sigma)$ be the maximal solution from Proposition~\ref{prop:loc well-posed Lp}. Then the following hold.
\begin{enumerate}
\item[(i)] (No trace limit at the explosion time.)
For every $T>0$,
$$
\mathbb{P}\left(\sigma<T \, \text{ and }\, \lim_{t\uparrow\sigma}Y(t)\ \text{exists in }X_{1-\frac{1+\kappa}{p},p}\right)=0.
$$
\item[(ii)] (Critical blow-up criterion.)
For every $T>0$,
$$
\mathbb{P}\left(\sigma<T \, \text{ and }\, 
\sup_{t\in[0,\sigma)} \|Y(t)   \|_{X_{1-\frac{1+\kappa}{p},p}}
+\| Y  \|_{L^p(0,\sigma;X_{1-\frac{\kappa}{p}})}<\infty\right)=0.
$$
Equivalently, on $\{\sigma<\infty\}$ we have almost surely
$$
\sup_{t\in[0,\sigma)} \|Y(t)\|_{X_{1-\frac{1+\kappa}{p},p}}
+\|Y\|_{L^p(0,\sigma;X_{1-\frac{\kappa}{p}})}=+\infty.
$$
\item[(iii)] (Subcritical blow-up criterion.)
Assume $(p, \kappa) $ is subcritical (in the sense that the criticality condition (4.3) in \cite{agresti2025nonlinear}) holds with strict inequality for every pair $(\beta_j, \rho_j)$ appearing in the non-linear estimate for $F$ and $G$). Then for every $T>0$,
$$
\mathbb{P}\left(\sigma<T  \text{ and } 
\sup_{t\in[0,\sigma)} \|Y(t) \|_{X_{1-\frac{1+\kappa}{p},p}}<\infty\right)=0,
$$
i.e. on $\{\sigma<\infty\}$ we have almost surely
$$
\sup_{t\in[0,\sigma)}\|Y(t)\|_{X_{1-\frac{1+\kappa}{p},p}}=+\infty.
$$
\end{enumerate}
\label{thm:blowup}
\end{thm}
\begin{proof}[Proof of Theorem~\ref{thm:blowup}]
This is a direct reformulation, in the present setting, of the abstract blow-up criteria in \cite[Theorems 5.1 and Theorem 5.2]{agresti2025nonlinear}.
\end{proof}
\begin{rem}
The criterion in Theorem~\ref{thm:blowup}(ii)-(iii) is typically used as follows: if one can
prove \emph{a priori} that on $\{\sigma<T\}$ the quantity
$$
\sup_{t\in[0,\sigma)}\|Y(t)\|_{X_{1-\frac{1+\kappa}{p},p}}
\quad\text{(and in the critical case also }\|Y\|_{L^p(0,\sigma;X_{1-\frac{\kappa}{p}})}\text{)}
$$
is finite almost surely, then $\mathbb{P}(\sigma<T)=0$ for every $T>0$, hence $\sigma=\infty$ a.s.
\label{rem:blowup_use}
\end{rem}

\section{Regularisation in time}\label{sec:regularisation}

\noindent 
In Section \ref{sec: blow-up criteria} we proved blow-up criteria for the maximal solution
$(Y,\sigma)$ from Proposition~\ref{prop:loc well-posed Lp}.
In particular, Theorem~\ref{thm:blowup}(i) states that at the maximal existence time $\sigma$
the process $Y$ cannot have a limit in the critical trace space $X_{1-\frac{1+\kappa}{p},p}$.
We now show that this ``no-trace-limit'' property forces \emph{instantaneous time regularisation}:
for every $t>0$ the solution enjoys improved time integrability and fractional time regularity.
This will be the starting point of the bootstrap arguments in the next section.

\begin{prop}[Instantaneous time regularisation for $\kappa>0$]
\label{prop:time-regularisation-kappa-pos}
Let $s\in[1,2)$, $p,q\in(2,\infty)$, and let
$\kappa\in(0,\frac p2-1)$. Assume the hypotheses of
Proposition~\ref{prop:loc well-posed Lp}. Let $(Y,\sigma)$ be the
maximal local solution from Proposition~\ref{prop:loc well-posed Lp}. Then, almost surely, for every $r\in(2,\infty)$, every
$\theta\in[0,\frac12)$, and every $\varepsilon\in(0,\theta)$,
\begin{align}
Y &\in H^{\theta,r}_{\loc}((0,\sigma);X_{1-\theta}),
\label{eq: regularity Hteta r}\\
Y &\in C^{\theta-\varepsilon}_{\loc}((0,\sigma);X_{1-\theta}).
\label{eq:regularity_holder}
\end{align}
\end{prop}

\begin{proof}
This is a direct application of the abstract parabolic regularisation
result \cite[Theorem~5.6]{agresti2025nonlinear} to the equation
\eqref{eq:MLC0}. Indeed, the proof of Proposition \ref{prop:loc well-posed Lp} verifies
the hypotheses of \cite[Theorem~4.7]{agresti2025nonlinear} in the
$(p,\kappa)$-framework. Moreover, Proposition~\ref{prop:SMR} gives
$$
(A,B) \in \mathcal{SMR}^{\bullet}_{r,\alpha}
\qquad \text{for every } r\in(2,\infty),\quad
\alpha\in\Big[0,\frac{r}{2}-1\Big),
$$
and Remark \ref{rem:lipgrowth} records that the nonlinear estimates
required in \cite[Assumption 4.1]{agresti2025nonlinear} are satisfied
for the present equation and are stable under the change of parameters
used in the proof of \cite[Theorem~5.6]{agresti2025nonlinear}.

Therefore all assumptions of \cite[Theorem 5.6]{agresti2025nonlinear}
are satisfied, and the conclusion of that theorem gives exactly
\eqref{eq: regularity Hteta r} and \eqref{eq:regularity_holder}.
\end{proof}

\begin{prop}[Instantaneous time regularisation for $\kappa=0$]
\label{prop:time regularisation kappa0}
Let $s\in[1,2)$ and $p,q\in(2,\infty)$. Assume the hypotheses of
Proposition~\ref{prop:loc well-posed Lp} with $\kappa=0$, and let
$(Y,\sigma)$ be the corresponding maximal local solution. Then, almost surely,
for every $r\in(2,\infty)$, every $\theta\in[0,\frac12)$, and every
$\varepsilon\in(0,\theta)$,
\begin{align}
Y &\in H^{\theta,r}_{\loc}((0,\sigma);X_{1-\theta}),
\label{eq:reg kappa 0 H theta}\\
Y &\in C^{\theta-\varepsilon}_{\loc}((0,\sigma);X_{1-\theta}).
\label{eq:reg kappa 0 holder}
\end{align}
\end{prop}

\begin{proof}
This is a direct application of \cite[Theorem 5.7]{agresti2025nonlinear}. Indeed, Proposition \ref{prop:loc well-posed Lp} verifies the hypotheses of
\cite[Theorem 4.7]{agresti2025nonlinear} in the $(p,0)$-framework and yields
the maximal $L^p_0$-solution $(Y,\sigma)$. Moreover,
Proposition~\ref{prop:SMR} gives
$$
(A,B) \in \mathcal{SMR}^{\bullet}_{r,\alpha} \qquad \text{for all } r \in (2,\infty), \quad
\alpha \in [0, \frac{r}{2}-1).
$$
It remains only to check the additional growth assumption
\cite[(5.21)]{agresti2025nonlinear}. In the notation of
\cite[Assumption~4.1]{agresti2025nonlinear}, Remark~\ref{rem:lipgrowth}
identifies the two pairs
$$
(\beta_1,\rho_1) = (  \beta , 1 ), \qquad
(\beta_2 , \rho_2 ) = (\hat\beta,2k),
$$
and proves the estimate
$$
\|F(x)\|_{X_0}
+ \|G(x)\|_{\gamma(\mathcal{U},X_{1/2})}
\leq C ( 1+\|x\|_{X_{\beta,1}}^2 +\|x\|_{X_{\hat\beta,1}}^{2k+1}),
\qquad x\in X_1.
$$
This is stronger than \cite[(5.21)]{agresti2025nonlinear}: one can choose
$C_n=C$ for every $n$, since the right-hand side is independent of any
bound on the trace norm
$\|x\|_{X_{1-\frac{1}{p},\infty}}$.

Therefore all assumptions of \cite[Theorem 5.7]{agresti2025nonlinear}
are satisfied, and its conclusion gives
\eqref{eq:reg kappa 0 H theta} and \eqref{eq:reg kappa 0 holder}.
\end{proof}

\section{Bootstrap in space}\label{sec:bootstrap-space}
\noindent 
In this section we first keep the Sobolev index $s$ fixed and improve only the spatial
integrability exponent $q$. The argument is based on the positive-time regularisation
proved in Section \ref{sec:regularisation}. Indeed, for every $\varepsilon>0$, on the event $\{\sigma>\varepsilon\}$ the random value $Y(\varepsilon)$ belongs to a better space than the original trace space, and this
allows us to restart the equation in an $L^p$-$L^{q_*}$ framework with $q_*>q$.
By pathwise uniqueness, the restarted solution coincides with the original one.

Let $(Y,\sigma)$ be the maximal local solution from Proposition \ref{prop:loc well-posed Lp},
with
$$
Y=( \vel,  \dir).
$$
Fix $q_*\in ( q, \infty)$ and define
$$
X_0^{(q_*)} : = \mathbb{H}^{ -s, q_*}( \T ^d) \times H^{1-s,q_*} (\T ^d ; \R ^d),
\qquad
X_1^{ (q_*) }:= \mathbb{ H }^{ 2-s, q_*}( \T^d)\times H^{3- s, q_*}( \T^d; \R^d),
$$
and, for $\theta \in [0,1]$,
$$
X_\theta^{ (q_* )}: = \mathbb{H}^{-s+2 \theta, q_*}(\T^d)\times H^{1-s+2 \theta, q_* }( \T^d ; \R^d ).
$$
Recall that the trace space is given in \eqref{eq:tracespaceLp}, i.e. 
$$
X_{1- \frac {1+\kappa} {p} ,p}^{(q_*)}
=
\mathbb{B}^{-s + 2 - \frac{2(1+ \kappa)}{p}}_{q_*,p}(\T^d; \R ^d) \times \rB^{-s+ 3-\frac{ 2( 1+\kappa)}{p}}_{ q_* , p }(\T^d ; \R^d).
$$

\begin{prop}[Bootstrap in $L^q$]\label{prop: bootstrap in Lq}
Let $s\in[1,2)$, $p,q\in(2,\infty)$, and
$\kappa\in[0,\frac p2-1)$.
Assume the hypotheses of Proposition~\ref{prop:loc well-posed Lp}.
Let $(Y,\sigma)$ be the maximal local solution. Let $ q_* > q$ be such that Proposition \ref{prop:loc well-posed Lp} remains valid with $q$ replaced by $q_*$. In particular, assume
\begin{align}
1-\frac{1+\kappa}{p} & \leq \frac{d-(s-1)q_*}{4q_*}+ \frac{s}{2} \leq 1 - \frac{1+\kappa}{2p}, \label{eq:bootstrap-q-subcrit1}\\
\frac{1}{k} & > s - 3+ \frac{d}{q_*}+2\left(\frac{1+\kappa}{p}\right), \label{eq:bootstrap-q-subcrit2}\\
q_* & \leq \frac{d}{s-1}, \quad \text{ if } s >1. \label{eq:bootstrap-q-subcrit3}
\end{align}
Assume moreover that there exists $\beta \in(0, \min \{\frac{1}{2}, \frac{1+\kappa}{p}\})$ such that
\begin{equation}\label{eq:bootstrap-q-window}
0<\frac{d}{q}-\frac{d}{q_*}<2 \left( \frac{1+\kappa}{p}- \beta \right).
\end{equation}
Then almost surely,
$$
Y \in L^p_{\loc}(0,\sigma;X_1^{(q_*)})
\cap C((0,\sigma);X_{1-\frac{1+\kappa}{p} , p}^{ (q_*)}).
$$
In particular,
$$
\vel \in L^p_{\loc}(0,\sigma;  \mathbb{H}^{2-s,q_*}( \T^d)), \qquad \dir \in L^p_{\loc}(0, \sigma; H^{3-s,q_*}( \T^d; \R^d)).
$$
\end{prop}

\begin{proof}
Fix $\varepsilon > 0$ and consider the event
$$
V_\varepsilon := \{ \sigma > \varepsilon\}.
$$
By Proposition~\ref{prop:time-regularisation-kappa-pos} if $\kappa>0$,
and by Proposition~\ref{prop:time regularisation kappa0} if $\kappa=0$, for every sufficiently small
$\eta>0$, we have
$$
Y\in C_{\loc}^{ \beta - \eta}((0,\sigma);X_{1 - \beta})
\qquad  \text{a.s.}
$$
Hence
$$
1_{ V_\varepsilon } Y( \varepsilon) \in X_{1-\beta} = \mathbb{H}^{-s+2-2\beta , q}( \T^d) \times H^{3-s-2\beta, q}( \T^d; \R^d)  \qquad  \text{a.s.}
$$
Next, \eqref{eq:bootstrap-q-window} guarantees the embedding
\begin{equation}
X_{1-\beta} \hookrightarrow X_{1-\frac{1+\kappa}{p},p}^{(q_*)}.
    \label{eq: embedding for X 1 beta}
\end{equation}
Indeed, for the velocity component one has
$$
-s+2-2\beta-\frac{d}{q} > -s+ 2- \frac{2(1+\kappa)}{p}-  \frac{d}{q_*},
$$
and the director component satisfies the same estimate, shifted by one derivative.
Therefore,
$$
1_ {V_\varepsilon} Y(\varepsilon) \in X_ {1-\frac{1+\kappa}{p}, p }^{ (q_*) }
\qquad \text{a.s.}
$$
We now restart \eqref{eq:MLC0} at time $\varepsilon$ in the $(p,\kappa,q_*)$-framework.
By Proposition \ref{prop:loc well-posed Lp} applied with $q_*$ in place of $q$, there exists a unique maximal local solution $(Z,\tau_*)$ on $[\varepsilon,\infty)$ with
initial value
$$
Z(\varepsilon) = 1_{V_\varepsilon} Y( \varepsilon),
$$
such that almost surely
$$
Z \in L^p_{\loc}( \varepsilon, \tau_*; X_1^{(q_*) })
\cap \rC ([\varepsilon,\tau_*); X_{ 1-\frac{1+\kappa}{p},p}^{(q_*)}).
$$
Since $q_*>q$, we have the embeddings
$$
X_1^{(q_*)}\hookrightarrow X_1,  \qquad X_{1-\frac{1+\kappa}{p},p}^{(q_*)}  \hookrightarrow X_{1-\frac{1+\kappa}{p} , p}.
$$
Hence $Z$ is also a local solution in the original $(p,\kappa,q)$-framework started at $\varepsilon$. By maximality of the maximal local solution, we have
$$
Y=Z \qquad\text{a.s. on }V_\varepsilon \times [ \varepsilon, \sigma \wedge \tau_*).
$$
It remains to prove that $\tau_*=\sigma$ a.s. on $V_\varepsilon$.

First, assume that $\mathbb{P}(V_\varepsilon \cap \{\sigma<\tau_*\}) > 0 $. Then the continuity of $Z$ in
$X_{1-\frac{1+\kappa}{p},p}^{(q_*)}$ yields a limit of $Y$ at time $\sigma$ in
$X_{1-\frac{1+\kappa}{p},p}$, contradicting Theorem \ref{thm:blowup}(i). Hence $\sigma<\tau_*$ is impossible. Conversely, assume that $\mathbb{P}(V_\varepsilon\cap\{\tau_*<\sigma \} )>0$. Since $\tau_*<\sigma$, Proposition \ref{prop:time-regularisation-kappa-pos}
if $\kappa>0$, and Proposition \ref{prop:time regularisation kappa0}
if $\kappa=0$, imply that
\begin{equation}
Y(\tau_*) \in  X_{1-\beta}
\qquad \text{a.s. on }V_\varepsilon\cap\{\tau_*<\sigma\} .
    \label{eq: Y tau star in X 1 minus beta}
\end{equation}
By the embedding in \eqref{eq: embedding for X 1 beta}, we have
$$
Y( \tau_* ) \in X_{1 - \frac{1+\kappa}{p},p}^{ ( q_* )}
\qquad  \text{a.s. on } V_\varepsilon \cap \{\tau_*<\sigma\}.
$$
Using the continuity of $Y$ in $X_{1-\beta}$ at $\tau_*<\sigma$, the embedding \eqref{eq: embedding for X 1 beta}, and the equality $Y=Z$ before $\tau_*$, we obtain
$$
\lim_{ t \uparrow \tau_*}  Z(t) =Y(\tau_*)
 \quad \text{in } X_{1 - \frac{1+\kappa}{p},p}^{(q_*)}
$$
a.s. on $V_\varepsilon \cap \{\tau_*<\sigma\}$.
This contradicts Theorem \ref{thm:blowup}(i), applied with $q$ replaced by $q_*$.
Hence $\tau_*<\sigma$ is impossible as well.

Therefore $\tau_*=\sigma$ a.s. on $V_\varepsilon$. Since $\varepsilon>0$ was arbitrary,
the claim follows.
\end{proof}

\begin{rem}[How to choose the exponent $q_*$]\label{rem:choose-qstar}
Fix $\beta \in (0, \min\{ \frac{1}{2}, \frac{ 1+ \kappa}{p} \} )$. If
$$
1-\frac{1+\kappa}{p} < \frac{d-(s-1) q}{ 4q } + \frac{s}{2}
\qquad  \text{and}  \qquad
 q < \frac{d}{s-1}\quad\text{if }s>1,
$$
with no additional upper restriction when $s=1$, then one may choose $q_*>q$ sufficiently close to $q$ so that
\eqref{eq:bootstrap-q-subcrit1}-\eqref{eq:bootstrap-q-subcrit3} remain valid and
$$
0 < \frac{d}{q} - \frac{d}{q_*}< 2\left(\frac{1+\kappa}{p}-\beta \right).
$$
Hence, if the lower inequality in \eqref{eq:subcrit1} is strict and $q<\frac{d}{s-1}$ when $s>1$, then this argument gives some gain in spatial integrability. If $q=\frac{d}{s-1}$ with $s>1$, or if the lower inequality in \eqref{eq:subcrit1} is an equality, the present argument does not give a larger exponent $q_*$. Equality in the upper inequality is not an obstruction.
\end{rem}

\begin{rem}\label{rem:bootstrap-Lq-kappa0}
In the case $ \kappa = 0$, the admissible gain is smaller: one must choose $q_*>q$ sufficiently
close to $q$ so that
$$
0 < \frac{d}{q} - \frac{d}{q_*} < \frac{2}{p}.
$$
Since $\beta$ may be chosen arbitrarily small in Proposition \ref{prop:time regularisation kappa0},
this still yields a nontrivial bootstrap step in spatial integrability whenever a choice of $q_*>q$
compatible with the admissibility conditions is available.
\end{rem}

We now discuss the Sobolev parameter. A genuine gain in Sobolev regularity corresponds to
\emph{decreasing} the parameter $s$: if $s_*<s$, then $\mathbb{H}^{2-s_*,q}$ and $H^{3-s_*,q}$
are strictly more regular than $\mathbb{H}^{2-s,q}$ and $H^{3-s,q}$. 

Fix $s_*\in[1,s)$ and define
$$
X_0^{(s_*)} := \mathbb{H}^{-s_*,q}( \T^d) \times H^{1-s_*,q}(\T^d;\R^d), \qquad X_1^{(s_*)} := \mathbb{H}^{2-s_*,q}(\T^d) \times  H^{3-s_*,q}(\T^d;\R^d),
$$
and, for $\theta \in [0,1]$
$$
X_ \theta^{ (s_*) } := \mathbb{H}^{-s_*+2\theta,q} ( \T^d ) \times H^{1-s_*+2\theta,q} ( \T^d;\R^d).
$$
From \eqref{eq:tracespaceLp}, the corresponding trace space is
$$
\rX_{1-\frac{1+\kappa}{p},p}^{(s_*)} =
\mathbb{B} ^ {- s_* +  2- \frac{ 2 (1+\kappa)}{p}}_{q,p}( \T^d; \R^d)
\times \rB^{ -s_*+3-\frac{2(1+\kappa)}{p}}_{q,p}( \T^d;\R^d).
$$

\begin{prop}[Bootstrap in Sobolev regularity]
\label{prop:bootstrap s case}
Assume the hypotheses of Proposition \ref{prop:loc well-posed Lp}, and assume in addition that $\kappa\in[0,\frac{p}{2}-1)$. Let $(Y, \sigma)$ be the maximal local solution from Proposition \ref{prop:loc well-posed Lp}. Let $s_*\in[1,s)$ be such that the hypotheses of Proposition \ref{prop:loc well-posed Lp} remain valid with $s$ replaced by $s_*$.
In particular, assume
\begin{align}
1-\frac{1+\kappa}{p} &\leq \frac{d-(s_*-1)q}{4q}+\frac{s_*}{2}
\leq 1-\frac{1+\kappa}{2p}, \label{eq:bootstrap-s-subcrit1}\\
\frac{1}{k} &> s_*-3+\frac{d}{q}+2\Big(\frac{1+\kappa}{p}\Big), \label{eq:bootstrap-s-subcrit2}\\
q &\leq \frac{d}{s_*-1}, \quad \text{ if } s_* >1.
\label{eq:bootstrap-s-subcrit3}
\end{align}
Assume moreover that there exists $\beta\in ( 0 , \min\{\frac{1}{2},\frac{1+\kappa}{p}\})$ such that
\begin{equation}\label{eq:bootstrap-s-window}
0 < s-s_* < 2 \left(\frac{1+\kappa}{p}-\beta\right).
\end{equation}
Then almost surely,
$$
Y \in L^p_{\loc} (0, \sigma;X_1^{(s_*)})
\cap \rC ((0,\sigma); X_{1-\frac{1+\kappa}{p},p}^{ (s_*)}).
$$
Equivalently,
$$
\vel \in L^p_{\loc} (0,\sigma;\mathbb{H}^{2-s_*,q} ( \T^d) ),
\qquad \dir \in L^p_{\loc}(0,\sigma;H^{3-s_*,q}(\T^d;\R^d) ).
$$
\end{prop}

\begin{proof}
Fix $\varepsilon>0$ and work on the event $
V_\varepsilon := \{\sigma>\varepsilon\}.$ By Proposition \ref{prop:time-regularisation-kappa-pos} (if $\kappa >0$) and by Proposition \ref{prop:time regularisation kappa0} (if $\kappa = 0$), for every sufficiently small
$\eta>0$ one has
$$
Y\in \rC_{\loc}^{ \beta - \eta}  ((0,\sigma) ; X_{1-\beta})   \qquad  \text{a.s.}
$$
Hence
$$
1_{V_\varepsilon} Y (\varepsilon) \in X_{1-\beta}
= \mathbb{H}^{-s+2-2\beta,q}(\T^d) \times H^{ 3-s-2\beta, q}(\T^d;\R^d)
\qquad\text{a.s.}
$$
Next, \eqref{eq:bootstrap-s-window} ensures the embedding
$$
X_{1-\beta} \hookrightarrow X_{1-\frac{1+\kappa}{p},p}^{(s_*)}.
$$
Indeed, for the velocity component one has
$$
-s+2-2\beta>-s_*+2-\frac{2(1+\kappa)}{p},
$$
and the director component satisfies the same estimate, shifted by one derivative. The embedding $H^{\rho_1,q}(\T^d)\hookrightarrow \rB^{\rho_2}_{q,p}(\T^d)$ holds whenever $\rho_1>\rho_2$ and $1<p,q<\infty$. Applying it to the preceding strict inequality and to the director component, we obtain
$$
1_{V_\varepsilon}Y(\varepsilon) \in X_{1-\frac{1+\kappa}{p},p}^{(s_*)}
 \qquad\text{a.s.}
$$
We now restart \eqref{eq:MLC0} at time $\varepsilon$ in the $(p,\kappa,s_*)$-framework.
By Proposition \ref{prop:loc well-posed Lp} applied with $s_*$ in place of $s$,
there exists a unique maximal local solution $(Z,\tau_*)$ on $[\varepsilon,\infty)$ with
initial value
$$
Z(\varepsilon)=1_{V_\varepsilon}Y(\varepsilon),
$$
such that almost surely
$$
Z\in L^p_{\loc}(\varepsilon,\tau_*;X_1^{(s_*)})
\cap C([\varepsilon,\tau_*);X_{1-\frac{1+\kappa}{p},p}^{(s_*)}).
$$
Since $s_*<s$, we have the continuous embeddings
$$
X_1^{(s_*)}\hookrightarrow X_1,
\qquad
  X_{1-\frac{1+\kappa}{p},p}^{(s_*)}  \hookrightarrow X_{1-\frac{1+\kappa}{p},p}.
$$
Hence $Z$ is also a local solution in the original $(p,\kappa,s)$-framework started at $\varepsilon$.
By pathwise uniqueness in Proposition \ref{prop:loc well-posed Lp},
$$
Y=Z
  \qquad\text{a.s. on }V_\varepsilon\times[\varepsilon,\sigma\wedge\tau_*).
$$
It remains to prove that $\tau_*=\sigma$ a.s. on $V_\varepsilon$. First, assume that $\mathbb{P}(V_\varepsilon\cap\{\sigma<\tau_*\})>0$.
Then the continuity of $Z$ in
$X_{1-\frac{1+\kappa}{p},p}^{(s_*)}$ yields a limit of $Y$ at time $\sigma$ in
$X_{1-\frac{1+\kappa}{p},p}$, contradicting Theorem \ref{thm:blowup}(i).
Hence $\sigma<\tau_*$ is impossible.

Conversely, assume that $\mathbb{P}(V_\varepsilon\cap\{\tau_*<\sigma\})>0$.
Since $\tau_*<\sigma$, Proposition \ref{prop:time-regularisation-kappa-pos} (if $\kappa >0$) and Proposition \ref{prop:time regularisation kappa0} (if $\kappa = 0$), implies that
$$
Y(\tau_*)\in X_{1-\beta}
\qquad\text{a.s. on }V_\varepsilon\cap\{\tau_*<\sigma\}.
$$
By the embedding above,
\begin{equation}
X_{1-\beta}\hookrightarrow X_{1-\frac{1+\kappa}{p},p}^{(s_*)},
\label{eq: aux embedding in trace space}
\end{equation}
and therefore
$$
Y(\tau_*)\in X_{1-\frac{1+\kappa}{p},p}^{(s_*)}
\qquad\text{a.s. on }V_\varepsilon\cap\{\tau_*<\sigma\}.
$$
Using
$$
Y=Z \qquad\text{a.s. on }V_\varepsilon\times[\varepsilon,\tau_*),
$$
together with the continuity of $Y$ in $X_{1-\beta}$ at $\tau_*<\sigma$ and the embedding \eqref{eq: aux embedding in trace space}, we obtain
$$
\lim_{t\uparrow\tau_*} Z(t)=Y(\tau_*)
  \quad\text{in }X_{1-\frac{1+\kappa}{p},p}^{(s_*)}
$$
a.s. on $V_\varepsilon\cap\{\tau_*<\sigma\}$.
This contradicts Theorem \ref{thm:blowup}(i), applied with $s$ replaced by $s_*$. Hence $\tau_*<\sigma$ is impossible as well. Therefore $\tau_*=\sigma$ a.s. on $V_\varepsilon$. Since $\varepsilon>0$ was arbitrary,
the claim follows.
\end{proof}

\begin{rem}[How to choose the exponent $s_*$]\label{rem:choose-sstar}
Fix $\beta \in (0, \min\{\frac{1}{2},\frac{1+\kappa}{p}\})$.
Assume that $s>1$ and that the lower bound in \eqref{eq:subcrit1} is strict, i.e.
$$
1-\frac{1+\kappa}{p} < \frac{d-(s-1)q}{4q} + \frac{s}{2}.
$$
Then one may choose $s_*\in[1,s)$ sufficiently close to $s$ so that
\eqref{eq:bootstrap-s-subcrit1}-\eqref{eq:bootstrap-s-subcrit3} remain valid and
$$
0<s-s_*<2\left(\frac{1+\kappa}{p}-\beta\right).
$$
Thus the argument yields a nontrivial gain in Sobolev regularity whenever there is room to move $s$ downward.
If $s=1$, or if the lower bound in \eqref{eq:subcrit1} is saturated, this particular argument does not produce such a gain within the regime of Proposition \ref{prop:loc well-posed Lp}.
\end{rem}

\appendix

\section{Kadlec's formula}

\noindent
We state and prove a Kadlec's formula for Neumann boundary conditions below. In particular, we use a result from [Chap 5, Taylor Vol.\ 1], of which Theorem 1.3 announces a similar result with Dirichlet boundary conditions. 
\begin{lemma}[Kadlec's formula with Neumann boundary conditions] \label{lemma:Kadlec}\hfill\\
    Let $\Dom \subset \R^d$ be a convex smooth bounded domain. For all $f \in \{g \in H^2(\Dom, \R) | \partial_\nu g = 0 \text{ on } \partial\Dom\}$,
    \begin{equation*}
        \|(\nabla \otimes \nabla) f\|_{L^2(\Dom,\R^{d \times d})} \leq \|\Delta f\|_{L^2(\Dom,\R)}.
    \end{equation*}
    In particular, this shows that the inner product  $(\cdot , \cdot)_V$ defined in \eqref{eq: def inn prod V} is equivalent to the standard inner product on $H^1(\Dom;\R^d)\times H^2(\Dom;\R^d)$.
\end{lemma}

\begin{proof}
First we prove that $C_N^\infty := \{f \in C^\infty(\overline\Dom, \R) | \partial_\nu f = 0 \text{ on } \partial\Dom\}$ is dense in $H_N^2 := \{f \in H^2(\Dom, \R) | \partial_\nu f = 0 \text{ on } \partial\Dom\}$, both equipped with the norm $\|\cdot\|_{H^2(\Dom, \R)}$. Let $f \in H_N^2$. By density of $C^\infty(\overline\Dom, \R)$, there exists a sequence $\phi_n \in C^\infty(\overline\Dom, \R)$ such that $\phi_n \rightarrow f$ in $H^2(\Dom, \R)$. Consequently, $\nabla \phi_n \rightarrow \nabla f$ in $H^1(\Dom, \R^d)$, and then $\partial_\nu \phi_n \rightarrow \partial_\nu f = 0$ in $H^{1/2}(\partial\Dom, \R)$. We denote by $\tilde{\phi}_n$ the solutions to the following elliptic problems,
\begin{equation*}
    \Delta \tilde{\phi}_n = \tilde{\phi}_n, \quad \partial_\nu \tilde{\phi}_n = \partial_\nu \phi_n \text{ on } \partial \Dom.
\end{equation*}
By standard arguments, we infer that $\|\tilde{\phi}_n\|_{H^1(\Dom, \R)} + \|\Delta \tilde{\phi}_n\|_{L^2(\Dom, \R)} \rightarrow 0$, and then, by \cite[Proposition 7.5]{Taylor_1996}, we have $\tilde{\phi}_n \rightarrow 0$ in $H^2(\Dom, \R)$. Moreover, $\tilde{\phi}_n \in C^\infty(\overline\Dom, \R)$ by induction on the previous proposition. Therefore, we have $\hat{\phi}_n := \phi_n - \tilde{\phi}_n \in C_N^\infty$ by construction, and $\hat{\phi}_n \rightarrow f$ in $H_N^2$, which proves density.

\bigskip

\noindent
Now, let $f \in C_N^\infty$. Thus, by integration by parts,
\begin{equation*}
    \|(\nabla \otimes \nabla) f \|_{L^2(\Dom,\R^{d \times d})}^2 = \int_{\Dom} (\Delta f)^2\d \Dom + \sum_{i,j = 1}^d \int_{\partial \Dom} \nu_i \partial_j f \partial_{ij} f \: d\partial\Dom,
\end{equation*}
since $\sum_{i =1}^d \nu_i \partial_i f = \partial_\nu f = 0$ on $\partial \Dom$. Then, we rewrite the second term on the RHS,
\begin{align*}
    \sum_{i,j = 1}^d \int_{\partial \Dom} \nu_i \partial_j f \partial_{ij} f \: d\partial\Dom &= \sum_{i,j = 1}^d \int_{\partial \Dom} \partial_j f \partial_j (\nu_i \partial_i f) - (\partial_j f) (\partial_i f) (\partial_j \nu_i) \: d\partial\Dom\\
    &= \int_{\partial \Dom} \nabla f \cdot \nabla ( \partial_\nu f) - (\nabla f) \cdot J_\nu \cdot (\nabla f) \: d\partial\Dom,
\end{align*}
Here, since the outward unit normal $\nu$ is initially defined only
on $\partial\Dom$, we extend it to points sufficiently close
to the boundary by setting $\nu=\nabla\rho$, where $\rho$
is the signed distance to $\partial\Dom$ (positive outside
$\Dom$ and negative inside).
This consists in assigning to each such point the outward
unit normal at its nearest boundary point.
We denote the Jacobian of this extension by
$J_\nu=D\nu=D^2\rho$. Now, denote by $\partial_\tau f := \nabla f - (\partial_\nu f) \nu$. Hence, it holds that
\begin{equation*}
    \int_{\partial \Dom} \nabla f \cdot \nabla ( \partial_\nu f) \: d\partial\Dom = \int_{\partial \Dom} (\partial_\nu f) \partial_\nu ( \partial_\nu f) + \partial_\tau f \cdot \partial_\tau ( \partial_\nu f) \: d\partial\Dom =0,
\end{equation*}
since $\partial_\nu f = \partial_\tau ( \partial_\nu f) = 0$ on $\partial \Dom$. Therefore, 
\begin{align*}
    \|(\nabla \otimes \nabla) f \|_{L^2(\Dom,\R^{d \times d})}^2 &= \|\Delta f \|_{L^2(\Dom,\R)}^2 - \int_{\partial \Dom} (\nabla f) \cdot J_\nu \cdot (\nabla f) \: d\partial\Dom \leq \|\Delta f \|_{L^2(\Dom,\R)}^2,
\end{align*}
because $J_\nu$ has nonnegative eigenvalues, since $\Dom$ is convex. As this holds for all $f \in C_N^\infty$, we conclude by density in $H_N^2$. That is the result we seek.
\end{proof}

\noindent 
{\bf Acknowledgements. }{\small Gianmarco Del Sarto and Tarek Z\"{o}chling acknowledge the support from the DFG project FOR~5528. Esm\'ee Theewis is supported by the VICI subsidy VI.C.212.027 of the Netherlands Organisation for Scientific Research (NWO). Antoine Moneyron acknowledges the support of the ERC project 856408-STUOD. 

\printbibliography

\end{document}